\documentclass{jssc}

\usepackage{graphicx}%
\usepackage{multirow}%
\usepackage{amsmath,amssymb,amsfonts}%
\usepackage{amsthm}%
\usepackage{mathrsfs}%
\usepackage[title]{appendix}%
\usepackage{xcolor}%
\usepackage{textcomp}%
\usepackage{manyfoot}%
\usepackage{booktabs}%
\usepackage{algorithm}%
\usepackage{algorithmicx}%
\usepackage{algpseudocode}%
\usepackage{listings}%
\usepackage{url}%
\usepackage{cleveref}%
\usepackage{easyReview}
\usepackage{textcomp}

\usepackage{color}
\definecolor{keywordcolor}{rgb}{0.7, 0.1, 0.1}   
\definecolor{tacticcolor}{rgb}{0.0, 0.1, 0.6}    
\definecolor{commentcolor}{rgb}{0.4, 0.4, 0.4}   
\definecolor{symbolcolor}{rgb}{0.0, 0.1, 0.6}    
\definecolor{sortcolor}{rgb}{0.1, 0.5, 0.1}      
\definecolor{attributecolor}{rgb}{0.7, 0.1, 0.1} 

\def\textsubscript#1%
{$_{\text{#1}}$}

\def\cdd{\mbox{\boldmath$\cdot$}~}

\input{amssym.def}
\include{graphix}

\newcommand{\ZZ}{\mathbb{Z}}

\makeatletter
\def\@oddfoot{\hfill}
\newcount\shumeicount
\def\setshumei#1#2#3{%
  \shumeicount=\count0
  \def\@oddhead{%
    \raise-5pt\hbox to0pt{\vrule width\hsize height 0pt depth 0.4pt\hss}\relax
    \ifnum \shumeicount=\count0
      \raise-7pt\hbox to0pt{\vrule width\hsize height 0pt depth 0.4pt\hss}\relax
      #1
    \else
      \ifodd\count0
        #2
      \else
        #3
       \fi
     \fi
  }%
}
\makeatother
\makeatletter
\def\@oddfoot{\hfill}
\newcount\shujiaocount
\def\setshujiao{%
  \shujiaocount=\count0
  \def\@oddfoot{%
      \ifodd\count0
      \else
      \fi
  }%
}
\makeatother
\def\title#1#2#3#4{{
  \vspace*{0.3cm}
  \begin{flushleft} \Large\bf #1\end{flushleft}
  \vspace*{-0.2cm}
      \begin{flushleft}
      \bf #2
      \end{flushleft}
      \footnotetext{\hspace{-6mm} #3\\ #4}}}

\def\dshm#1#2#3#4
{\setshumei{J Syst Sci Complex (#1) #2:
{\thepage--\pageref{LastPage}}\hfill}
            {\hfill {\small #3}\hfill\hbox to0pt{\hss\thepage}}
            {\hbox to0pt{\thepage\hss}\hfill {\small #4}\hfill
            }
            \setshujiao}
\def\drd#1#2
{{\vskip 1cm\small \begin{flushleft}
 #1 \\
 #2\\
\copyright The Editorial Office of JSSC \&  Springer-Verlag GmbH
Germany 2021
\end{flushleft}}}

\def\epsilon{\varepsilon}

\begin{document}

\title{A Formal Proof of the Irrationality of $\zeta(3)$ in Lean 4}
{\uppercase{Liu} Junqi \cdd \uppercase{Zhang} Jujian \cdd \uppercase{Zhi} Lihong}
{\uppercase{Liu} Junqi \cdd \uppercase{Zhi} Lihong\\
State Key Laboratory of Mathematical Sciences, Academy of Mathematics and Systems Science, University of Chinese Academy of Science.  Email: liujunqi@amss.ac.cn; lzhi@mmrc.iss.ac.cn\\   
\uppercase{Zhang} Jujian \\
Department of Mathematics, Imperial College London. Email: jujian.zhang19@imperial.ac.uk \\
   } 
{$^*$This research was supported by the National Key R$\&$D Program of China 2023YFA1009401 and the Strategic Priority Research Program of Chinese Academy of Sciences under Grant XDA0480501.}

\drd{DOI: }{Received: August 8 2025}


\dshm{2025}{XX}{A Formal Proof of the Irrationality of $\zeta(3)$ in Lean 4}{\uppercase{Liu Junqi} $\cdd$ \uppercase{Zhang Jujian} $\cdd$
\uppercase{Zhi Lihong}}

\Abstract{
We formalize a proof of the irrationality of $\zeta(3)$ in Lean~4, using Beukers’ method. To support this, we extend the Lean mathematical library (\texttt{Mathlib}) by formalizing shifted Legendre polynomials and important results in analytic number theory that were previously missing. As part of the Lean~4  \texttt{PrimeNumberTheoremAnd} project, we also formalize the asymptotic behavior of the prime counting function, giving the first formal proof in Lean~4 of a version of the Prime Number Theorem with an error term which is stronger than what had previously been formalized. This result is a crucial ingredient in proving the irrationality of \(\zeta(3)\). Our complete Lean~4 formalization is publicly available on GitHub. \footnote{See \url{https://github.com/ahhwuhu/zeta_3_irrational}.}}




\section{Introduction}\label{sec1}

The Riemann zeta function is a crucial concept in mathematics. For real values of $s$ with $s > 1$, the Riemann zeta function is defined as
\[\zeta(s) := \sum\limits_{n=1}^{\infty} \frac{1}{n^s}.\]
In 1978, Ap\'ery proved that $\zeta(3)$ is irrational
\cite{ap1979ery}.
This result was the first dent in the problem of the irrationality of the values of the Riemann zeta function at odd positive integers; see \cite{van1979proof} for an informal report on Ap\'ery's proof by Van der Poorten.
In 1979, this proof was shortened by Beukers \cite{beukers1979note}, who used an integral method to connect $\zeta(3)$ with a specific double improper integral over the unit square $(0,1)^2$.

A computer algebra based formal proof of the irrationality of $\zeta(3)$ using the Coq proof assistant was  given by Salvy \cite{ BrunoSalvy2003},   Chyzak et al. \cite{chyzak2014computer} and Mahboubi and Sibut-Pinote
\footnote{\url{https://github.com/coq-community/apery}}\cite{mahboubi2021formal}. It follows  Ap\'ery's original proof \cite{ap1979ery} and uses the method of creative telescoping \cite{ zeilberger1991method}  for dealing with recurrence relations that appeared in  Ap\'ery's proof.

 Eberl formalized  Beukers' proof using the  Isabelle proof assistant \cite{Zeta_3_Irrational-AFP} based on the lecture notes of Filaseta \cite{Filaseta2011}. The asymptotic upper bound on ${\rm lcm} \{1, \ldots, n\} \leq O(c^n)$ for any $c > e$ (Euler's number) used by both Ap\'ery and  Beukers is available in Isabelle \cite{Prime_Number_Theorem-AFP}. 

Our work contributes to the mathematical library  \texttt{Mathlib} \cite{mathlib} for the Lean 4 theorem  proof assistant \cite{moura2021lean}, a system based on dependent type theory augmented with quotient types and classical reasoning. \texttt{Mathlib} is a decentralized and continuously evolving library, with contributions from over 340 authors. While Lean’s library excels in many areas, it has lagged behind other theorem provers, such as Isabelle, in certain analytical domains. Our project aims to bridge this gap by formalizing important theorems in areas like calculus and analytic number theory, thereby enhancing Lean’s analytical content and further enriching \texttt{Mathlib}'s diverse body of work.

Although Lean trails behind Isabelle in formalizing some foundational theorems, such as the Prime Number Theorem, significant progress is being made. Terence Tao, Alex Kontorovich, and others are actively working on the \texttt{PrimeNumberTheoremAnd} project~\footnote{\url{https://github.com/AlexKontorovich/PrimeNumberTheoremAnd}} in Lean 4. As part of this effort, we have formalized related results — Theorem 25 and Corollary 9~\footnote{\url{https://alexkontorovich.github.io/PrimeNumberTheoremAnd/web/sect0004.html}}—which are crucial for proving the irrationality of $\zeta(3)$. Additionally, during this process, we identified and corrected a typo in their formal theorem statements.

Our formalization of the proof of the irrationality of $\zeta(3)$ in Lean 4 follows mainly Beukers' method \cite{beukers1979note}.
A key idea in Beukers' proof for showing that a real number $x$ is irrational is to construct a non-zero sequence $\{a_n + b_n x\}$, where $a_n, b_n \in \ZZ$, that tends to zero as $n \rightarrow \infty$.   If $x$ were rational, say  $x=\frac{p}{q}, q >0$,   the sequence $\{|a_n + b_n x|\}$ would have a lower bound of $\frac{1}{q}$, independent from $n$,  leading to a contradiction. 
Our main steps are outlined by referring to Beukers' proof. 
In all the following multiple integrals, we denote by $\mu$ a measure on the region of integration, and use $d\mu$ to indicate differential element.

\begin{itemize}
    \item  
     Consider  the integral 
\begin{equation}\label{integral6}
\int_{(x,y) \in (0,1)^2}  -P_n(x)P_n(y)\frac{\log (x y)}{1-x y} d\mu 
\end{equation}

where $P_n(x)$ is the shifted Legendre polynomial
\begin{equation}
 P_n(x) := \frac{1}{n!}\frac{d^n}{dx^n}[x^n(1-x)^n]. 
 \end{equation}
According to Lemma 1 in \cite{beukers1979note}, the  integral (\ref{integral6}) equals to 
 $\frac{a_n + b_n \zeta(3)}{d_n^3}$, where
 \begin{equation}
 d_n := {\rm lcm} \{1, 2, \ldots, n\},
  \end{equation}
  the least common multiple  of $1,2, \ldots, n$ and $a_n, b_n \in \mathbb{Z}$.

\item According to the Prime Number Theorem \cite{hadamard1896distribution,de1896recherches,ireland2013classical},
for sufficiently large values of $n$, we have
$$
d_n^3 \leq  \left(e^3\right)^n.
$$
 
Besides, the integral (\ref{integral6}) is positive and  bounded above by $2 \left(\frac{1}{24}\right)^n \zeta(3)$. Hence, for sufficiently large $n$, one has 
\[0 < |a_n + b_n \zeta(3)|< 2 \left(\frac{1}{24}\right)^n \zeta(3) d_n^3  <\left(\frac{21}{24}\right)^n 2 \zeta(3) ,\]
    which implies the irrationality of $\zeta(3)$. 
\end{itemize}

In this paper, we make the following main contributions:
\begin{itemize}

\item We introduce and rigorously define shifted Legendre polynomials, formalizing key properties within Lean 4, thus advancing the formalization of special functions in the Lean ecosystem.

\item We provide the first formal proof in Lean 4 of a strengthened version of the Prime Number Theorem, establishing that the prime counting function \(\pi(x)\) is asymptotic to \(\frac{x}{\log x}\).  This key result (Corollary 9 in the \texttt{PrimeNumberTheoremAnd} project) enhances Lean’s capabilities in analytic number theory.

\item  We present a complete formal proof of the irrationality of $\zeta(3)$ in Lean 4, following Beukers' method, contributing to the formal verification of an important result in analytic number theory.

\end{itemize}

To enhance readability, we provide both formal and informal statements of the definitions and theorems. The proofs of the theorems are primarily presented informally, with the corresponding formal proofs available in Lean 4, which can be accessed and downloaded on GitHub\footnote {\url{https://github.com/ahhwuhu/zeta_3_irrational}}. 

The paper is organized as follows. In Section \ref{sec2}, we introduce the concept of Lebesgue integrability in Lean 4, focusing on the lower Lebesgue integral \lstinline{lintegral}, which is more convenient for dealing with multiple integrals. In Section \ref{sec3}, we formally define the shifted Legendre polynomials and establish their fundamental properties in Lean 4. Section \ref{sec4} presents two key results related to the Prime Number Theorem. Finally,  in Section \ref{sec5}, we formally show the irrationality of $\zeta(3)$ in Lean 4.

\section{Lean Preliminaries}\label{sec2}

Lean is an open-source theorem prover with a small trusted kernel based on dependent type theory \cite{de2015lean}. One of its most exciting applications is in training large language models (LLMs) for theorem proving, leveraging Lean 4's formal framework to enable AI systems to assist in automated reasoning and proof generation \cite{yang2023leandojo,song2024largelanguagemodelscopilots}.  Google DeepMind has translated one million problems written in natural language into Lean, without including human-written solutions, for training AlphaProof to solve International Mathematical Olympiad problems at a silver medalist level. By formalizing  the proof of irrationality of $\zeta(3)$ in Lean 4, we
 aim to add some knowledge in the fields of analysis, combinatorics, and number theory to Lean's mathematical library.

As the formal proof of the irrationality of  $\zeta(3)$
 is closely tied to demonstrating that certain improper integrals converge, we begin by introducing basic definitions related to the integrability of functions in \texttt{Mathlib}.

In \texttt{Mathlib}, the integrability of a function $f$ is defined using the concept of the measurability of the function and its Lebesgue integrability over a given domain. The formal definition of integrability of $f$ in  Lean 4 is:

\begin{lstlisting}[frame=single]
def MeasureTheory.Integrable {α} {_ : MeasurableSpace α} (f : α → β)
   (μ : Measure α := by volume_tac) : Prop :=
   AEStronglyMeasurable f μ ∧ HasFiniteIntegral f μ
\end{lstlisting}

The above definition involves two key parts: measurability and integrability. 
The formal definition of measurability in Lean 4 is:

\begin{lstlisting}[frame=single]
def MeasureTheory.AEStronglyMeasurable {_ : MeasurableSpace α} (f : α → β)
   (μ : Measure α := by volume_tac) : Prop := ∃ g, StronglyMeasurable g ∧ f =ᵐ[μ] g
\end{lstlisting}

A function is \lstinline{AEStronglyMeasurable} if it is almost everywhere equal to the limit of a sequence of simple functions. A simple function is a measurable function whose image consists of only a finite set of real numbers, and any simple function can be expressed as a linear combination of a finite number of characteristic functions \cite{hytonen2016analysis}.

The functions considered in this paper are elementary functions, which are functions generated by a finite number of basic operations such as addition, multiplication, inversion, and composition involving basic functions like polynomial functions, rational functions, exponential functions, logarithmic functions, and trigonometric functions. Elementary functions are always measurable, and since measurable functions can be approximated by simple functions, elementary functions are always \lstinline{AEStronglyMeasurable}. Hence, the functions we are considering will be \lstinline{AEStronglyMeasurable} as well.

A function $f$ is Lebesgue integrable if its Lebesgue lower integral over the domain is finite. 
The formal definition of a function having a finite integral in Lean 4 is: 
\begin{lstlisting}[frame=single]
def MeasureTheory.HasFiniteIntegral {_ : MeasurableSpace α} (f : α → β)
   (μ : Measure α := by volume_tac) : Prop := (∫⁻ a, ‖f a‖₊ ∂μ) < ∞
\end{lstlisting}

The ``$\int^-$'' symbol in the definition is a notation for \lstinline{lintegral}, mathematically the lower Lebesgue integral of a $[0, \infty]$ valued function. The lower Lebesgue integral of a function is obtained by approximating the integral from below by 
simple functions and take the infimum of the integrals of those approximating functions. 

The lower Lebesgue integral \lstinline{lintegral} is the extended real-valued version of the integral. In  Lean 4, the set of extended non-negative real numbers $[0, \infty]$   is defined as \lstinline{ENNReal}. According to dependent type theory, the objects on both sides of an equation should have the same type; otherwise, the equation would be ill-typed. Therefore, in all equations appearing later in this article, if one side of the mathematical expression is a
\lstinline{lintegral}, which means the ``$\int^-$'' symbol appears on one side, and the other side is implicitly assumed to be of the type  \lstinline{ENNReal}.
There are two functions \verb|ENNReal.ofReal| and \verb|ENNReal.toReal| which can be used to convert types between non-negative real numbers and \lstinline{ENNReal}.  

To prove the integral of  a function is finite,  we compute the lower Lebesgue integral 
  \lstinline{lintegral}, and check whether it is finite. 

Compared to the standard Lebesgue integral, \lstinline{integral}, the benefit of using \lstinline{lintegral} is that issues of integrability or summability do not arise at all. We only need to calculate the specific \lstinline{lintegral} value.

We can connect the lower Lebesgue integral and integral by the following theorem: 

\begin{lstlisting}[label={lst:integral_eq_lintegral_of_nonneg_ae}, caption={the lower Lebesgue integral and integral}, frame=single]
theorem MeasureTheory.integral_eq_lintegral_of_nonneg_ae {f : α → ℝ} 
    (hf : 0 ≤ᵐ[μ] f) (hfm : AEStronglyMeasurable f μ) :
    ∫ a, f a ∂μ = ENNReal.toReal (∫⁻ a, ENNReal.ofReal (f a) ∂μ) 
\end{lstlisting}

If $f$ is an elementary function, then $f$ must be \verb|AESstronglyMeasurable|. The condition $0\leq^{\mathsf{m}}$[$\mu$] $f$ means that $f$ is  non-negative almost everywhere under the measure $\mu$, i.e., the set of points where 
$f(x)<0$ has measure zero. We can prove it through the following theorem in \texttt{Mathlib}:

\begin{lstlisting}[label={ae_nonneg_restrict_of_forall_setIntegral_nonneg_inter}, caption={function nonnegative almost everywhere}, frame=single]
theorem MeasureTheory.ae_nonneg_restrict_of_forall_setIntegral_nonneg_inter 
    {f : α → ℝ} {t : Set α} (hf : IntegrableOn f t μ) (hf_zero : 
    ∀ s, MeasurableSet s → μ (s ∩ t) < ⊤ → 0 ≤ ∫ x in s ∩ t, f x ∂μ) : 
    0 ≤ᵐ[μ.restrict t] f
\end{lstlisting}

To apply theorem  \cref{lst:integral_eq_lintegral_of_nonneg_ae}, we need to check two conditions: \verb|hf_zero| and \lstinline{IntegrableOn}. 
The condition \verb|hf_zero| can be checked easily due to the non-negativity of the function at every point in  $(0,1)^2$. 
Below, we focus on checking the \lstinline{IntegrableOn} condition.

The condition \verb|IntegrableOn|  is satisfied if a function is integrable with respect to the restricted measure on a set \verb|s|, essentially, it is to prove that the value of \lstinline{lintegral} is finite.  

\begin{lstlisting}[frame = single]
theorem MeasureTheory.hasFiniteIntegral_iff_norm (f : α → β) :
    HasFiniteIntegral f μ ↔ (∫⁻ a, ENNReal.ofReal ‖f a‖ ∂μ) < ∞
\end{lstlisting}

The following  theorem shows  that  \verb|ENNReal.ofReal| and \verb|ENNReal.toReal| can be cancelled out:

\begin{lstlisting}[frame = single]
theorem ENNReal.toReal_ofReal_eq_iff {a : ℝ} : (ENNReal.ofReal a).toReal = a ↔ 0 ≤ a 
\end{lstlisting}

In Lean 4, we can relate an integral in the \texttt{ENNReal} space to a 
\texttt{lintegral} using the following theorem, which allows us to move  \verb|ENNReal.ofReal|
 from outside to inside the \texttt{lintegral}. Formally, the theorem is stated as:

\begin{lstlisting}[label = {lst:ofReal_integral_eq_lintegral_ofReal}, caption = {\texttt{ENNReal} of integral equal to \texttt{lintegral}}, frame = single]
theorem MeasureTheory.ofReal_integral_eq_lintegral_ofReal {f : α → ℝ}
    (hfi : Integrable f μ) (f_nn : 0 ≤ᵐ[μ] f) :
    ENNReal.ofReal (∫ x, f x ∂μ) = ∫⁻ x, ENNReal.ofReal (f x) ∂μ
\end{lstlisting}

Additionally, in the proof of Lemma \ref{eqn:double_integral_eq2}, we will also make use of the substitution formula for integrals:

\begin{lstlisting}[label = {change_of_variables}, caption = {change of variables}, frame = single]
theorem intervalIntegral.integral_comp_mul_deriv {f f' g : ℝ → ℝ}
    (h : ∀ x ∈ uIcc a b, HasDerivAt f (f' x) x)
    (h' : ContinuousOn f' (uIcc a b)) (hg : Continuous g) :
    (∫ x in a..b, (g ∘ f) x * f' x) = ∫ x in f a..f b, g x
\end{lstlisting}

To prove that a multiple integral is equal to  a repeated integral, we use the following theorem in \texttt{Mathlib}:

\begin{lstlisting}[label = {multiple_integral_equal_repeated_integral}, caption = {multiple integral equal to repeated integral}, frame = single]
theorem MeasureTheory.integral_prod (f : α × β → E) (hf : Integrable f (μ.prod ν)) :
    ∫ z, f z ∂μ.prod ν = ∫ x, ∫ y, f (x, y) ∂ν ∂μ 
\end{lstlisting}

The following theorem in \texttt{Mathlib} states that for a non-negative integral function $f$, the integral of $f$ is positive if and only if the measure of the support of $f$ is positive. 

\begin{lstlisting}[label = {lst:integral_pos_iff_support_of_nonneg_ae}, caption = {integral positive iff support nonnegetive almost everywhere},frame = single]
theorem MeasureTheory.integral_pos_iff_support_of_nonneg_ae {f : α → ℝ}
    (hf : 0 ≤ᵐ[μ] f) (hfi : Integrable f μ) :
    (0 < ∫ x, f x ∂μ) ↔ 0 < μ (Function.support f) 
\end{lstlisting}

\section{Shifted Legendre Polynomial}\label{sec3}
The shifted  Legendre polynomials
$$ P_n(x) := \frac{1}{n!}\frac{d^n}{dx^n}[x^n(1-x)^n] $$
have been used in Beukers' proof for constructing a convergent sequence $\{a_n + b_n \zeta(3)\}$. 

In this section, we formally define the shifted Legendre polynomial and outline its fundamental properties in Lean 4. These definitions and properties have been added to \texttt{Mathlib}.

In Lean 4, the shifted  Legendre polynomial is defined as 
\begin{lstlisting}[frame = single]
noncomputable def shiftedLegendre (n : ℕ) : ℝ[X] :=
  C (n ! : ℝ)⁻¹ * derivative^[n] (X ^ n * (1 - X) ^ n)
\end{lstlisting}
where \lstinline{C} is the embedding of $\mathbb R$ into its polynomial ring and \lstinline{X} represents the variable.

By expanding the polynomial  $(x-x^2)^n$, 
and  combining it with the linearity of the derivative operator,  the shifted Legendre polynomials $P_n(x)$ can be written as polynomials with integer coefficients:
\begin{equation}\label{eqn:legendre_int}
P_n(x)=\sum\limits_{k=0}^{n}(-1)^k\binom{n}{k}\binom{n+k}{n}x^k,
\end{equation}
which is formalized as the following theorem in Lean 4:
\begin{lstlisting}[frame = single]
theorem shiftedLegendre_eq_sum (n : ℕ) : shiftedLegendre n =
    ∑ k in Finset.range (n + 1), C ((- 1) ^ k : ℝ) *
    (Nat.choose n k : ℝ[X]) * (Nat.choose (n + k) n : ℝ[X]) * X ^ k 
\end{lstlisting}

One can prove a more abstract version of the above theorem, which generalizes the result as follows:

\begin{lstlisting}[frame = single]
lemma shiftedLegendre_eq_int_poly (n : ℕ) : ∃ a : ℕ → ℤ, shiftedLegendre n =
    ∑ k in Finset.range (n + 1), (a k : ℝ[X]) * X ^ k 
\end{lstlisting}

The shifted Legendre polynomials have good properties for performing integration by parts. For all $n \in \mathbb{N}$ and a  $\mathcal{C}^n$ functions $f : [0, 1]\rightarrow \mathbb{R}$, we have:
\begin{equation}
\int_{0}^{1}P_n(x)f(x) \, dx = \frac{(-1)^n}{n!} \int_{0}^{1} x^n(1-x)^n\frac{d^nf}{dx^n} \, dx
\label{eqn:shifited-lengendre-poly-integral}
\end{equation}

We present  the formalization of equation (\ref{eqn:shifited-lengendre-poly-integral}) for 
\[f (y): = \frac{1}{1-(1-x y)z},\]
which will be used in Lemma \ref{eqn:double_integral_eq1} and Lemma \ref{eqn:double_integral_eq3}, with $0<x,z<1$.

\begin{lemma}\label{n_derivative}
For $0<x,z<1$, we have 
\begin{equation}
\frac{d^n}{dy^n}\left(\frac{1}{1-(1-x y)z}\right) = (-1)^n n!\frac{(xz)^n}{(1-(1-x y)z)^{n+1}}.
\label{eqn:n_derivative}
\end{equation}
\end{lemma}

The formal statement in Lean 4 is: 
\begin{lstlisting}[frame = single]
lemma n_derivative' {x z : ℝ} (n : ℕ) (hx : x ∈ Set.Ioo 0 1) (hz : z ∈ Set.Ioo 0 1) : 
  (deriv^[n] fun y ↦ 1 / (1 - (1 - x * y) * z)) =
  (fun y ↦ (-1) ^ n * n ! * (x * z) ^ n / (1 - (1 - x * y) * z) ^ (n + 1))
\end{lstlisting}

\begin{proof} When we formalize equality (\ref{eqn:n_derivative}) in Lean 4,   we need to discuss whether $1-(1-xy)z$ is equal to $0$ for $x,y,z\in(0, 1)$.
\begin{itemize}
    \item For any $y$, when $1-(1-xy)z \neq 0$,  $\frac{1}{1-(1-xy)z}$ is differentiable to the $n$-th order by induction.  It is straightforward to check the equality (\ref{eqn:n_derivative}).
    
    \item   If $1-(1-xy)z = 0$, then, because functions with a zero denominator are defined as $0$ in Lean 4 (since all functions are total), the values on the right-hand side of equality (\ref{eqn:n_derivative}) are zero.
     In Lean 4, the derivative at points where a function is not differentiable is defined as $0$. Hence,  we prove the left side of equality (\ref{eqn:n_derivative}) is $0$  by demonstrating, using an $\epsilon-\delta$ argument,   that $\frac{1}{1-(1-xy)z}$ is not differentiable at the point $y$ where  $1-(1-xy)z=0$.
      Consequently, the equality (\ref{eqn:n_derivative}) holds with a value of $0$.  
\end{itemize} \hfill $\qedsymbol$
\end{proof}

\begin{lemma}\label{int_shifted_legendre}
For $0<x,z<1$, one has
\begin{equation}
\int_0^1 P_n(y)\frac{1}{1-(1-x y)z} \, d y = \int_0^1  \frac{(x y z)^n(1-y)^n}{(1-(1- x y)z)^{n+1}} \, d y.
\label{eqn:shifited-lengendre-integral-special}
\end{equation}
\end{lemma}

The formal statement in Lean 4 is:
\begin{lstlisting}[frame = single]
lemma legendre_integral_special {x z : ℝ} (n : ℕ) (hx : x ∈ Set.Ioo 0 1)
    (hz : z ∈ Set.Ioo 0 1) : 
  ∫ (y : ℝ) in (0)..1, 
    eval y (shiftedLegendre n) * (1 / (1 - (1 - x * y) * z)) =
  ∫ (y : ℝ) in (0)..1, 
    (x * y * z) ^ n * (1 - y) ^ n / (1 - (1 - x * y) * z) ^ (n + 1)
\end{lstlisting}

\begin{proof} It can be proven by induction and integration by parts that
\[ \int_0^1 P_n(y)\frac{1}{1-(1-x y)z} \, dy = \frac{(-1)^n}{n!} \int_0^1 y^n(1-y)^n \frac{d^n}{dy^n}\left(\frac{1}{1-(1-x y)z}\right) \, dy \]

By substituting equality (\ref{eqn:n_derivative}) into the right-hand side of the above equation, we obtain equality (\ref{eqn:shifited-lengendre-integral-special}). \hfill $\qedsymbol$
\end{proof}

\section{Prime Number Theorem}\label{sec4}

Suppose the prime factorization of $d_n = \operatorname{lcm}\{1,..., n\}$ is 
$$
d_n = \prod_{p\leq n} p^m,  \,\, {\rm where}~ m = \max_{k \in \mathbb{N}} \{p^k \leq n\}.
$$
Since $\log$ is strictly monotonic, we have \[m = \left \lfloor \frac{\log n}{\log p} \right \rfloor.\]
We have the following estimate for $d_n$:
$$
d_n = \prod_{p\leq n} p^{\lfloor \frac{\log n}{\log p} \rfloor} \leq \prod_{p\leq n} p^{\frac{\log n}{\log p}} = \prod_{p\leq n} p^{\log_{p} n} = \prod_{p\leq n} n = n^{\pi (n)},
$$
where $\pi(n)$ is the number of primes less than or equal to  $n$.

 The Prime Number Theorem states that:
\[
\lim_{x \to \infty} \frac{\pi(x)}{x / \log x} = 1.
\]

There are several equivalent ways to express the Prime Number Theorem. One of the most common is:
\[  \lim_{x \to \infty} \frac{\pi\left(x\right)}{{ \int_{2}^{x}\frac{dt}{\log t}} } = 1. 
\]


The Prime Number Theorem was first proved in 1896 by Jacques Hadamard \cite{hadamard1896distribution} and by Charles de la Vallée Poussin \cite{de1896recherches} independently. A modern proof was given by 
Atle Selberg and Paul Erdős, 1948, independently  \cite{ireland2013classical}. 
 In \cite{avigad2007formally}, Avigad et al. presented the first formalization of Selberg's elementary proof \cite{selberg1949elementary}  of the Prime Number Theorem using the Isabelle/HOL prover \cite{nipkow2002isabelle}, which was later reproved in Metamath by Carneiro \cite{carneiro2016formalization}. Subsequently, Harrison provided a formal proof of the Prime Number Theorem based on Newman's presentation \cite{newman1980simple} using the HOL-Light prover \cite{harrison2009formalizing}. This work was later extended by Eberl and Paulson in Isabelle \cite{Prime_Number_Theorem-AFP}. Later, Song and Yao present a formalized version of the Prime Number Theorem with an explicit error term in Isabelle \cite{PNT_with_Remainder-AFP}.

Currently, no formal proof of the Prime Number Theorem exists in Lean. Terence Tao, Alex Kontorovich, and others are actively working on the \texttt{PrimeNumberTheoremAnd} 
project in Lean 4.  The goal of this project is to formalize the Prime Number Theorem in Lean, including a classical error term, along with several related results in analytic number theory. A long-term objective is the formalization of the {Chebotarev Density Theorem}—a fundamental theorem in algebraic number theory that generalizes the Prime Number Theorem to Galois extensions of number fields.

We prove Theorem \ref{pi_asymp} and Theorem \ref{pi_alt} corresponding to Theorem 25 and Corollary 9 in the \texttt{PrimeNumberTheoremAnd} project.~\footnote{\url{https://alexkontorovich.github.io/PrimeNumberTheoremAnd/web/sect0004.html}} These constitute the first formal proof in Lean of a strengthened version of the Prime Number Theorem incorporating an error term which is stronger than what had previously been formalized. Furthermore, Theorem \ref{pi_alt} plays a crucial role in the formal proof of the irrationality of \(\zeta(3)\). The formalization of   Theorem \ref{pi_asymp} and Theorem \ref{pi_alt} in Lean 4 has been contributed to the \texttt{PrimeNumberTheoremAnd} project and is publicly available.\footnote{\url{https://github.com/AlexKontorovich/PrimeNumberTheoremAnd/pull/211}}

\begin{theorem}\label{pi_asymp}
  The prime counting function admits the following asymptotic estimate as $ x \to \infty$
  \[
  \pi(x) = \left(1 + o(1)\right)\int_2^{x}\frac{1}{\log t} d t.
  \]
\end{theorem}

The formal statement in Lean 4 is:
\begin{lstlisting}[frame = single]
theorem pi_asymp :
    ∃ c : ℝ → ℝ, c =o[atTop] (fun _ ↦ (1 : ℝ)) ∧
    ∀ᶠ (x : ℝ) in atTop, Nat.primeCounting ⌊x⌋₊ = (1 + c x) * ∫ t in Set.Icc 2 x, 1 / (log t) ∂ volume
\end{lstlisting}

A precise description of the auxiliary constants involved in $o(1)$ and 
the concept of ``sufficiently large" is essential for formalization. Therefore, we present the proof in great detail, ensuring that each step can be easily transcribed into Lean 4.

\begin{proof} We aim to show that $\frac{\pi\left(x\right)}{{ \int_{2}^{x}\frac{1}{\log t}d t} }-1$
    is $o\left(1\right)$, that is, for every $\epsilon$, there exists
    $M_{\epsilon}\in\mathbb{R}$ such that for all $x >M_{\epsilon}$, we have
    \[\left|\frac{\pi\left(x\right)}{{ \int_{2}^{x}\frac{1}{\log t}d t}}-1\right|\le\epsilon.\]
    
    For all $ x \geq 2$, it has been  formalized in Lean 4 that:
    \[
    \pi\left(x\right)={\frac{1}{\log x}\sum_{p\le\lfloor x\rfloor}\log p+{\int_{2}^{x}\frac{\sum_{p\le\lfloor t\rfloor}\log p}{t\log^{2}t}}d t}.
    \]

    We also know that for every $\epsilon> 0$, there exists a function $f_{\epsilon}:\mathbb{R}\to\mathbb{R}$
    such that $f_{\epsilon}=o\left(\epsilon\right)$ and $f_{\epsilon}$ is integrable
    on $\left(2,x\right)$ for all $x \geq 2$. Furthermore,  for $x$ sufficiently
    large, say $x>N_{\epsilon}\ge2$, we have
    \[
    \sum_{p\le\lfloor x\rfloor}\log p=x+xf_{\epsilon}\left(x\right).
    \]
    Hence, for every  $\epsilon>0$ and for sufficiently
    large  $x$, such a function  $f_{\epsilon}$ satisfies 
    \[
    \pi\left(x\right)=\frac{x+xf_{\epsilon}\left(x\right)}{\log x}+\int_{2}^{N_{\epsilon}}\frac{\sum_{p\le\lfloor x\rfloor}\log p}{t\log^{2}t}d t+\int_{N_{\epsilon}}^{x}\frac{t+tf_{\epsilon}\left(t\right)}{t\log^{2}t}d t,
    \]
    which can be  simplified  to
    \[\pi\left(x\right)=\left(\frac{x}{\log x}+\int_{N_{\epsilon}}^{x}\frac{1}{\log^{2}t}d t\right)+\left(\frac{xf_{\epsilon}\left(x\right)}{\log x}+\int_{N_{\epsilon}}^{x}\frac{f_{\epsilon}\left(t\right)}{\log^{2}t}d t\right)+\int_{2}^{N_{\epsilon}}\frac{\sum_{p\le\lfloor x\rfloor}\log p}{t\log^{2}t}d t.
    \]
    Using integration by parts, we obtain
    \[
    \frac{x}{\log x}+\int_{N_{\epsilon}}^{x}\frac{1}{\log^{2}t}d t=\int_{N_{\epsilon}}^{x}\frac{1}{\log t}d t+\frac{N_{\epsilon}}{\log N_{\epsilon}}=\int_{2}^{x}\frac{1}{\log t}d t+\left(\frac{N_{\epsilon}}{\log N_{\epsilon}}-\int_{2}^{N_{\epsilon}}\frac{1}{\log t}d t\right).
    \]
    Hence
    \[
    \pi\left(x\right)=\int_{2}^{x}\frac{1}{\log t}d t+\left(\frac{xf_{\epsilon}\left(x\right)}{\log x}+\int_{N_{\epsilon}}^{x}\frac{f_{\epsilon}\left(t\right)}{\log^{2}t}d t\right)+C_{\epsilon},
    \]
   for some constant $C_{\epsilon}\in\mathbb{R}$. Therefore, we have 
    \[
    \frac{\pi\left(x\right)}{\int_{2}^{x}\frac{1}{\log t}d t}-1=\left(\frac{xf_{\epsilon}\left(x\right)}{\log x}+\int_{N_{\epsilon}}^{x}\frac{f_{\epsilon}\left(t\right)}{\log^{2}t}d t\right)/\int_{2}^{x}\frac{1}{\log t}d t+\frac{C_{\epsilon}}{\int_{2}^{x}\frac{1}{\log t}d t}.
    \]
    
    Recall that $f_{\epsilon}=o\left(\epsilon\right)$, so  for
    all $c>0$, there exists $M_{c,\epsilon}$ such that for all $x>M_{c,\epsilon}$, we have \[\left|f_{\epsilon}\left(x\right)\right|\le c\epsilon.\]
    Therefore, for $x>M_{c,\epsilon}>2$, we have
    \[
    \begin{aligned}\frac{xf_{\epsilon}\left(x\right)}{\log x} & \le\frac{c\epsilon\cdot x}{\log x}\\
    \left|\int_{N_{\epsilon}}^{x}\frac{f_{\epsilon}\left(t\right)}{\log^{2}t}d t\right| & \le\int_{N_{\epsilon}}^{M_{c,\epsilon}}\left|\frac{f_{\epsilon}\left(t\right)}{\log^{2}t}\right|d t+\int_{M_{c,\epsilon}}^{x}\left|\frac{f_{\epsilon}\left(t\right)}{\log^{2}t}\right|d t\\
     & \le\int_{N_{\epsilon}}^{M_{c,\epsilon}}\frac{\left|f_{\epsilon}\left(t\right)\right|}{\log^{2}t}d t+c\epsilon\int_{M_{c,\epsilon}}^{x}\frac{1}{\log^{2}t}d t\\
     & =\int_{N_{\epsilon}}^{M_{c,\epsilon}}\frac{\left|f_{\epsilon}\left(t\right)\right|}{\log^{2}t}d t+c\epsilon\left(\int_{M_{c,\epsilon}}^{x}\frac{1}{\log t}d t+\frac{M_{c,\epsilon}}{\log M_{c,\epsilon}}-\frac{x}{\log x}\right).
    \end{aligned}
    \] 
    Hence,  for $x >M_{c,\epsilon}>2$, we have
    \[
    \begin{aligned}\left|\frac{xf_{\epsilon}\left(x\right)}{\log x}+\int_{N_{\epsilon}}^{x}\frac{f_{\epsilon}\left(t\right)}{\log^{2}t}d t\right| & \le\int_{N_{\epsilon}}^{M_{c,\epsilon}}\frac{\left|f_{\epsilon}\left(t\right)\right|}{\log^{2}t}d t+c\epsilon\left(\int_{M_{c,\epsilon}}^{x}\frac{1}{\log t}d t+\frac{M_{c,\epsilon}}{\log M_{c,\epsilon}}\right)\\
     & =\int_{N_{\epsilon}}^{M_{c,\epsilon}}\frac{\left|f_{\epsilon}\left(t\right)\right|}{\log^{2}t}d t+c\epsilon\left(\int_{2}^{x}\frac{1}{\log t}d t+\frac{M_{c,\epsilon}}{\log M_{c,\epsilon}}-\int_{M_{c,\epsilon}}^{2}\frac{1}{\log t}d t\right).
    \end{aligned}
    \]
    Let $D_{c,\epsilon}$ denote the value of 
    \[\int_{N_{\epsilon}}^{M_{c,\epsilon}}\frac{\left|f_{\epsilon}\left(t\right)\right|}{\log^{2}t}d t+c\epsilon\frac{M_{c,\epsilon}}{\log M_{c,\epsilon}}-c\epsilon\int_{M_{c,\epsilon}}^{2}\frac{1}{\log t}d t,\] we observe that 
    \[
    \begin{aligned}\left|\frac{\pi\left(x\right)}{\int_{2}^{x}\frac{1}{\log t}d t}-1\right| & \le\left(c\epsilon\int_{2}^{x}\frac{1}{\log t}d t+D_{c,\epsilon}\right)/\int_{2}^{x}\frac{1}{\log t}d t+\frac{C_{\epsilon}}{\int_{2}^{x}\frac{1}{\log t}d t}\\
     & =c\epsilon+\frac{D_{c,\epsilon}}{\int_{2}^{x}\frac{1}{\log t}d t}+\frac{C_{\epsilon}}{\int_{2}^{x}\frac{1}{\log t}d t}.
    \end{aligned}
    \]
    In particular, for $c=\frac{1}{2}$, there exists a constant $D$, such that for all $x>\max\left(M_{\frac{1}{2},\epsilon},N_{\epsilon}\right)$, we have
    \[
    \left|\frac{\pi\left(x\right)}{\int_{2}^{x}\frac{1}{\log t}d t}-1\right|\le\frac{\epsilon}{2}+\frac{D}{\int_{2}^{x}\frac{1}{\log t}d t}.
    \]
    
    
    Note that 
    \[\int_{2}^{x}\frac{1}{\log t}d t\ge\frac{\left(x-2\right)}{\log x},\]
     for $x>e^{s}$,  $s>1$, we have the inequality: \[\int_{2}^{x}\frac{1}{\log t}d t\ge\frac{e^{s}-2}{s}.\]
     Consequently, for sufficiently large  $s>A_{\epsilon}>1$, we have \[\frac{D}{\int_{2}^{x}\frac{1}{\log t}d t}\le\frac{sD}{e^{s}-2}\le\frac{sD}{e^{s}}\le\frac{\epsilon}{2}. \]
    Thus, for all $x>\max\left(M_{\frac{1}{2},\epsilon},N_{\epsilon},e^{A_{\epsilon}}\right)$, we obtain
    \[\left|\frac{\pi\left(x\right)}{\int_{2}^{x}\frac{1}{\log t}d t}-1\right|\le\epsilon.\]
    This completes the proof that  $\frac{\pi\left(x\right)}{\int_{2}^{x}\frac{1}{\log t}d t}-1$
    is $o\left(1\right)$ for sufficiently large $x$. \hfill $\qedsymbol$
\end{proof}

Theorem \ref{pi_asymp} allows us to express the asymptotic distribution law of prime numbers as follows:
\begin{theorem}\label{pi_alt}
    The prime counting function $\pi(x)$ satisfies the asymptotic estimate
\[
\pi(x) = (1 + o(1)) \frac{x}{\log x},
\]
as $x \to \infty$.
\end{theorem}

The formal statement in Lean 4 is:
\begin{lstlisting}[frame = single]
theorem pi_alt : ∃ c : ℝ → ℝ, c =o[atTop] (fun _ ↦ (1:ℝ)) ∧
    ∀ x : ℝ, Nat.primeCounting ⌊x⌋₊ = (1 + c x) * x / log x
\end{lstlisting}

\begin{proof} There exists a constant 
$c_1$ such that for sufficiently large 
$x$. we have 
\[ \int_2^{\sqrt{x}} \frac{1}{(\log t)^2} \, d t \leq \frac{1}{(\log 2)^2}(\sqrt{x} - 2) \leq c_1\sqrt{x}.  \]

Similarly, there exists a constant 
$c_2$ such that for sufficiently large 
$x$, we have
\[ \int_{\sqrt{x}}^x \frac{1}{(\log t)^2} \, d t \leq \frac{1}{(\log \sqrt{x})^2}(x-\sqrt{x}) \leq \frac{1}{4(\log x)^2}x \leq c_2 \frac{x}{(\log x)^2}.\]

Since for sufficiently large 
$x$, $(\log x)^2 \leq \sqrt{x}$,  $\sqrt{x} \leq \frac{x}{(\log x)^2}$,  there exists a constant 
$c$ such that 
\[ \int_2^x \frac{1}{(\log t)^2} \, d t \leq  c_1 \frac{x}{(\log x)^2} + c_2 \frac{x}{(\log x)^2}\leq c \frac{x}{(\log x)^2}.\]

By integrating by parts, we obtain:
\begin{equation}\label{pi_alt_aux1}
    \int_2^x \frac{1}{\log t} \, d t= \frac{x}{\log x} - \frac{2}{\log 2} + \int_2^x \frac{1}{(\log t)^2}\, d t.
\end{equation} 
Let
\[ g(x) = \left(\int_2^x \frac{1}{(\log t)^2}\, d t - \frac{2}{\log 2} \right)\frac{\log x}{x}. \]
For sufficiently large $x$, we have
\begin{align*}
    |g| &= \left|\int_2^x \frac{1}{(\log t)^2}\, d t - \frac{2}{\log 2} \right|\frac{\log x}{x}\\
    &\leq \left|\int_2^x \frac{1}{(\log t)^2}\, d t\right| \frac{\log x}{x} + \left| \frac{2}{\log 2} \right|\frac{\log x}{x}\\
    &\leq c \frac{x}{(\log x)^2}\frac{\log x}{x} +  \left| \frac{2}{\log 2} \right|\frac{\log x}{x}\\
    &=  \frac{c}{\log x} +  \left| \frac{2}{\log 2} \right|\frac{\log x}{x}.
\end{align*}
When $x \rightarrow \infty$, 
we have $\frac{1}{\log x} \rightarrow 0$, and $ \frac{\log x}{x} \rightarrow 0$ as $x$ increases much faster than $\log x$. Hence, we have
\[g(x) = o(1).\]

By equation (\ref{pi_alt_aux1}), we obtain
\begin{align*}
    \int_2^x \frac{1}{\log t} \, d t &= \frac{x}{\log x} - \frac{2}{\log 2} + \int_2^x \frac{1}{(\log t)^2}\, d t \\
    &= \frac{x}{\log x} + \left(\left(\int_2^x \frac{1}{(\log t)^2}\, d t - \frac{2}{\log 2} \right)\frac{\log x}{x}\right)\frac{x}{\log x}\\
    &= \left(1 + g(x) \right)\frac{x}{\log x}.
\end{align*}

By Theorem \ref{pi_asymp}, there exists a function $f : \mathbb{R}\to\mathbb{R}$ such that $f=o\left(1\right)$, and for sufficiently large $x$,
\[ \pi(x) = \left(1 + f(x) \right)\int_2^{x}\frac{1}{\log t} d t. \]
For sufficiently large $x$, 
 we can complete the proof by replacing $ \int_2^x \frac{1}{\log t} \, d t$ by $\left(1 + g(x) \right)\frac{x}{\log x}$:
\begin{align*}
    \pi(x) &= \left(1 + f(x) \right)\int_2^{x}\frac{1}{\log t} d t\\
    &= \left(1 + f(x) \right)\left(1 + g(x) \right)\frac{x}{\log x}\\
    &= (1 + o(1))\frac{x}{\log x},
\end{align*}
as $f(x) + g(x) + f(x)g(x)= o(1).$ \hfill $\qedsymbol$
\end{proof}

By Theorem \ref{pi_alt}, for sufficiently large  $n$, we have
$$
d_n \leq n^{\pi (n)} \sim n^{\frac{n}{\log n}} = \left(e ^{\log n}\right)^{\frac{n}{\log n}} = e ^ n.
$$
Consequently, for sufficiently large $n$, we have 
\begin{equation}\label{upperbounddn}
d_n^3 \leq \left(e^n\right)^3 = \left(e^3\right)^n \leq 21^n.
\end{equation}
The upper bound (\ref{upperbounddn}) will be used to bound the sequence of $\{|a_n + b_n \zeta(3)|\}$.

\section{Formal Proof of Irrationality  of $\zeta(3)$ in Lean 4}\label{sec5}

We first consider an essential class of double integrals given in \cite[Lemma 1  (b)]{beukers1979note}. Let $r$ and $s$ be natural numbers. We define
\begin{equation}\label{eqJrs}
J_{rs} := \int_{(x,y)\in(0,1)^2} -\frac{\log (x y)}{1-x y}x^r y^s d\mu .
\end{equation}

The formal definition of the function $J_{rs}$ in  Lean 4 is:
\begin{lstlisting}[frame = single]
noncomputable abbrev J (r s : ℕ) : ℝ := ∫ (x : ℝ × ℝ) in Set.Ioo 0 1 ×ˢ Set.Ioo 0 1,
   -(x.1 * x.2).log / (1 - x.1 * x.2) * x.1 ^ r * x.2 ^ s
\end{lstlisting}

Since $-\frac{\log (x y)}{1-x y}x^r y^s$ is not negative for any $x, y \in (0, 1)$,  we can define
 the lower Lebesgue integral \verb|J_ENN_rs| in Lean 4:
\begin{lstlisting}[frame = single]
noncomputable abbrev J_ENN (r s : ℕ) : ENNReal := 
    ∫⁻ (x : ℝ × ℝ) in Set.Ioo 0 1 ×ˢ Set.Ioo 0 1,
    ENNReal.ofReal (- (x.1 * x.2).log / (1 - x.1 * x.2) * x.1 ^ r * x.2 ^ s)
\end{lstlisting}

For all $x, y \in (0, 1)$,  expanding $\frac{1}{1-xy}$ using the geometric series, we have
\begin{equation}\label{eq12}
-\frac{\log(x y)}{1-x y}x^r y^ s = \sum_{n \in \mathbb{N}} - \log(x y)x^{r+n} y^{s+n}.
\end{equation}

Furthermore, we aim to express 
$J_{rs}$ as $\sum_{n\in\mathbb{N}} \int_{(x,y) \in (0, 1)^2} -\log(xy) x^{n + r}y^{n+s}$. In other words, in Lean 4, we seek to prove the following equality  by interchanging the order of the lower Lebesgue integral $``\int^-"$, summation sign $\sum_{n \in \mathbb{N}}$ and the type conversion
\lstinline{ENNReal.ofReal}:

\begin{lstlisting}[frame = single]
J_ENN r s = ∑' (n : ℕ), ∫⁻ (x : ℝ × ℝ) in Set.Ioo 0 1 ×ˢ Set.Ioo 0 1,
    ENNReal.ofReal (- (x.1 * x.2).log * x.1 ^ (n + r) * x.2 ^ (n + s))
\end{lstlisting}

This is another example of the advantages of using \lstinline{lintegral} over \lstinline{integral}: the lower Lebesgue integral commutes with infinite sums without needing to check integrability or summability conditions. To interchange the order of summation and type conversion, we must verify convergence using the following theorem:
\begin{lstlisting}[frame = single]
theorem ENNReal.ofReal_tsum_of_nonneg {f : α → ℝ}
   (hf_nonneg : ∀ n, 0 ≤ f n) (hf : Summable f) :
   ENNReal.ofReal (∑' n, f n) = ∑' n, ENNReal.ofReal (f n)
\end{lstlisting}

\subsection{Linear Form of $J_{rs}$}

We dedicate this section to demonstrating the following theorem.
\begin{theorem}\label{maintheorem5.1}
The integral $J_{rs}$ (\ref{eqJrs}) can be expressed by the formula, 
\begin{equation}\label{form:Jrs}
J_{rs} = a_{rs}\zeta(3) + \frac{b_{rs}}{d_{\max\{r,s\}}^3},
\end{equation}
where $a_{rs}$ and $b_{rs}$ are integers, and $d_{\max\{r,s\}}= {\rm lcm} \{1, 2, \ldots, {\max\{r,s\}}\}$. 
\end{theorem}

The formal statement of the theorem  in Lean 4 is: 
\begin{lstlisting}[frame = single]
lemma linear_int_aux : ∃ a b : ℕ → ℕ → ℤ, ∀ r s : ℕ, J r s =
    b r s * ∑' n : ℕ, 1 / ((n : ℝ) + 1) ^ 3 + a r s / (d (Finset.Icc 1 (Nat.max r s))) ^ 3
\end{lstlisting}

The connection between $J_{rs}$ and $\zeta(3)$ has been recorded in the following lemmas given in \cite[Lemma 1]{beukers1979note}.

\begin{lemma} \label{thm5.2}
The integral $J_{rr}$ can be written as: 
\begin{equation}
J_{r r} = 2\zeta(3) - 2 \sum\limits_{m = 1}^{r}\frac{1}{m^3}.
\label{eqn:Jrr}
\end{equation}
we let $\sum\limits_{m = 1}^{r}\frac{1}{m^3}=0$ for  $r=0$.  In particular, we have $J_{00} = 2\zeta(3)$.
\end{lemma}

The formal statement in Lean 4 is: 
\begin{lstlisting}[frame = single]
theorem J_rr (r : ℕ) : J r r = 2 * ∑' n : ℕ, 1 / ((n : ℝ) + 1) ^ 3 -
    2 * ∑ m in Finset.Icc 1 r, 1 / (m : ℝ) ^ 3 
\end{lstlisting}

\begin{lemma}\label{thm5.3}
For $r,s \in \mathbb{N}$, assume $r \neq s$, we have
\begin{equation}
J_{rs} = \frac{\sum\limits_{m = 1}^{r}\frac{1}{m^2} - \sum\limits_{m = 1}^{s}\frac{1}{m^2}}{r - s}.
\label{eqn:Jrs}
\end{equation}

\end{lemma}

The formal statement in Lean 4 is: 
\begin{lstlisting}[frame = single]
theorem J_rs {r s : ℕ} (h : r ≠ s) : J r s =
    (∑ m in Icc 1 r, 1 / (m : ℝ) ^ 2 - ∑ m in Icc 1 s, 1 / (m : ℝ) ^ 2) / (r - s)
\end{lstlisting}

We will prove Lemma \ref{thm5.2} and Lemma \ref{thm5.3} later. For now, assume they are true, according to equality (\ref{eqn:Jrr}) and equality (\ref{eqn:Jrs}), one can show that for all distinct $r,s \in \mathbb{N}$, there exist integers $z_r$ and $z_{rs}$ such that
\begin{equation}\label{eqn:J_rr_rs}
J_{rr} = 2\zeta(3)-\frac{z_r}{d_r^3}~\text{ and }~J_{rs}=\frac{z_{rs}}{d_r^3}.
\end{equation}
where  $d_r := {\rm lcm} \{1, 2, \ldots, r\}$. 

By equalities (\ref{eqn:J_rr_rs}), we immediately derive a unified form of  $J_{rs}$
  (as given in equalities (\ref{form:Jrs})), thereby concluding the proof of Theorem \ref{maintheorem5.1}. 
 
 To establish the equalities (\ref{eqn:J_rr_rs}),  we begin by calculating a special family of lower Lebesgue integrals which:
$$\int^-_{(x,y)\in(0, 1)^2} -\log(xy) x^{k + r} y^{k+s} d\mu ,$$
for natural numbers $k,r,s$. That is, we formalize the following lemma:
\begin{lemma}\label{J_ENN_eq_aux}
\[
\int^-_{(x,y)\in(0, 1)^2}-\log(xy)x^{k+r}y
^{k+s} d\mu =
\frac{1}{(k+r+1)^2 (k+s+1)} + \frac{1}{(k+r+1)(k+s+1)^2} .
\]
\end{lemma}

The formal statement in Lean 4 is:
\begin{lstlisting}[frame = single]
lemma J_ENN_rs_eq_tsum_aux_intergal (r s k : ℕ) : 
    ∫⁻ (x : ℝ × ℝ) in Set.Ioo 0 1 ×ˢ Set.Ioo 0 1, ENNReal.ofReal (- (x.1 * x.2).log * x.1 ^ (k + r) * x.2 ^ (k + s)) = ENNReal.ofReal (1 / ((k + r + 1) ^ 2 * (k + s + 1)) + 1 / ((k + r + 1) * (k + s + 1) ^ 2))
\end{lstlisting}

\begin{proof} For lower Lebesgue integrals, the double integral is equivalent to the repeated integral without any assumption on integrability.

We regard $x$ as a parameter and integrate $y$ to get
\begin{align*}
    &\int^-_{(x,y)\in(0, 1)^2}-\log(xy)x^{k+r}y^{k+s} d\mu \\
    =& \int^-_{x\in(0, 1)}\int^-_{y\in(0, 1)}-\log(xy)x^{k+r}y^{k+s} \\
    =& \int^-_{x\in(0, 1)}\int^-_{y\in(0, 1)}\left(-\log(x)x^{k+r}y^{k+s} + \left(-\log(y)x^{k+r}y^{k+s}\right)\right) \\
    =& \int^-_{x\in(0, 1)}\int^-_{y\in(0, 1)}-\log(x)x^{k+r}y^{k+s} + \int^-_{x\in(0, 1)}\int^-_{y\in(0, 1)}-\log(y)x^{k+r}y^{k+s}.
\end{align*}

We  consider the following two special integrals, which can be directly calculated :
\[ \int^-_{x\in(0, 1)}-\log(x)x^n = \frac{1}{(n+1)^2}, ~{\rm and}~  ~~ \int^-_{x\in(0, 1)}x^n = \frac{1}{n+1}.    \]
The formal statements in Lean 4 are:
\begin{lstlisting}[frame = single]
lemma ENN_log_pow_integral (n : ℕ) : ∫⁻ (x : ℝ) in Set.Ioo 0 1,
    ENNReal.ofReal (-x.log * x ^ n) = ENNReal.ofReal (1 / (n + 1) ^ 2)
\end{lstlisting}
and 
\begin{lstlisting}[frame = single]
lemma ENN_pow_integral (n : ℕ) : ∫⁻ (x : ℝ) in Set.Ioo 0 1, ENNReal.ofReal (x ^ n) = ENNReal.ofReal (1 / (n + 1))
\end{lstlisting}
Using the above two lemmas twice, one can get
\begin{align*}
  &\int^-_{(x,y)\in(0, 1)^2}-\log(xy)x^{k+r}y^{k+s} d\mu \\
  =& \int^-_{x\in(0, 1)}-\log(x)x^{k+r}\frac{1}{k+s+1} + 
  \int^-_{x\in(0, 1)}x^{k+r}\frac{1}{(k+s+1)^2} \\
  =& \frac{1}{(k+r+1)^2 (k+s+1)} + \frac{1}{(k+r+1)(k+s+1)^2}.
\end{align*}

This proves Lemma \ref{J_ENN_eq_aux}. \hfill $\qedsymbol$
\end{proof}

Let us now proceed with the proofs of Lemma \ref{thm5.2} and Lemma \ref{thm5.3}. 
By equation (\ref{eq12}) and Lemma \ref{J_ENN_eq_aux}, we have
\begin{align}
    &\int^-_{(x,y)\in(0, 1)^2}-\frac{\log (x y)}{1-x y}x^r y^s d\mu \nonumber \\
    =&\sum_{k=0}^{\infty}\left(\frac{1}{(k+r+1)^2 (k+s+1)} + \frac{1}{(k+r+1)(k+s+1)^2} \right). \label{krs}
\end{align}

The formal statement in Lean 4 is:
\begin{lstlisting}[frame = single]
lemma J_ENN_rs_eq_tsum (r s : ℕ) : J_ENN r s = ∑' (k : ℕ), ENNReal.ofReal
    (1 / ((k + r + 1) ^ 2 * (k + s + 1)) + 1 / ((k + r + 1) * (k + s + 1) ^ 2))
\end{lstlisting}

In order to show Lemma \ref{thm5.2} and Lemma \ref{thm5.3}, we discuss the convergence of the series
\[
\sum_{k=0}^{\infty}\left(\frac{1}{(k+r+1)^2 (k+s+1)} + \frac{1}{(k+r+1)(k+s+1)^2}\right),
\]
in two cases: $r = s$ and $r \ne s$ corresponding to Lemma \ref{thm5.2} and Lemma \ref{thm5.3}.

\begin{proof}[Proof of Lemma \ref{thm5.2}] 

For $r = s$, the right side of equality (\ref{krs}) is equal to
\[
 2 \sum_{n = 0}^{\infty}\frac{1}{(n + 1)^3} - 2\sum_{m = 1}^{r} \frac{1}{m^3} = 2\zeta(3)-2\sum_{m=1}^{r}\frac{1}{m^3},
\]
which implies
\begin{equation}\label{eq_lst7}
    \int^-_{(x,y)\in(0, 1)^2}-\frac{\log (x y)}{1-x y}x^r y^s d\mu = 2\zeta(3)-2\sum_{m=1}^{r}\frac{1}{m^3}.
\end{equation}

The formal statement in Lean 4 is:

\begin{lstlisting}[label = {lst:J_ENN_rr}, caption={lintegral form of $J_{rr}$}, frame=single]
lemma J_ENN_rr (r : ℕ) : J_ENN r r = ENNReal.ofReal
    (2 * ∑' n : ℕ, 1 / ((n : ℝ) + 1) ^ 3 - 2 * ∑ m in Finset.Icc 1 r, 1 / (m : ℝ) ^ 3)
\end{lstlisting}

As $\sum_{n=1}^{\infty}\frac1{n^3}$ is convergent, the above series is convergent. Therefore, we have proved Lemma \ref{thm5.2}. \hfill $\qedsymbol$
\end{proof}

\begin{proof}[Proof of Lemma \ref{thm5.3}]
For $ r \neq s$, as the right side of equality (\ref{krs})  is symmetric with respect to $r$ and $s$, without loss of generality, we can assume $r > s$. By rewriting the right side of equality (\ref{krs}), we have
\[
  \begin{aligned}
    & \sum_{k \in \mathbb{N}} \left(\frac{1}{(k + r + 1) ^ 2(k + s + 1)} +\frac{1}{(k + r + 1)(k + s + 1)^2}\right)\\
    =& \sum_{k \in \mathbb{N}} \frac{1}{r - s}\left(\frac{1}{(k + s +1)^2} - \frac{1}{(k + r +1)^2}\right) \\
    =& \frac{1}{r - s}\sum_{k=s+1}^r \frac{1}{k^2}.
  \end{aligned}
\]
As  $\sum_{n=1}^{\infty}\frac1{n^2}$ is convergent, we know that the above series converges. We have then proved Lemma \ref{thm5.3}. \hfill $\qedsymbol$
\end{proof}

Now we proceed to prove Theorem \ref{maintheorem5.1}. 

\begin{proof} The expressions of  $2\zeta(3) - 2\sum_{k=1}^r \frac{1}{k^3}$ and $\frac{1}{r - s} (\sum_{k=1}^r \frac{1}{k^2}-\sum_{k=1}^s \frac{1}{k^2})$   are non-negative as they result from simplifying the non-negative series in   (\ref{krs}).

Let $d_r$ denote the least common multiple from $1$ to $r$.  As every number in $\{1,\dots,r\}$ is divisible by $d_r$, $\sum_{k=1}^r \frac{d_r^3}{k^3}, \sum_{k=1}^s \frac{d_r^2}{k^2}, \sum_{k=1}^r \frac{d_r^2}{k^2}$ and $\frac{d_r}{r - s}$ are integers. Hence, 
we have 
\[
  J_{rr}   d_r^3=  \left(2\zeta(3) - 2\sum_{k=1}^r \frac{1}{k^3}\right) d_r^3 = 2\zeta(3)d_r^3 - 2\sum_{k=1}^r \frac{d_r^3}{k^3}  \in 2\zeta(3)d_r^3 - \mathbb{Z}
\]
and for $r > s$,
\[
   J_{rs} d_r^3 =\left(\frac{1}{r - s} \left(\sum_{k=1}^r \frac{1}{k^2}-\sum_{k=1}^s \frac{1}{k^2}\right)\right)d_r^3 = \frac{d_r}{r - s} \left(\sum_{k=1}^r \frac{d_r^2}{k^2}-\sum_{k=1}^s \frac{d_r^2}{k^2}\right) \in \mathbb{Z}. 
\]

For the case of $r<s$, we can prove that $J_{rs} = J_{sr}$ by multiplying both the numerator and denominator of $J_{rs}$ in equality (\ref{eqn:Jrs}) by $-1$.

Therefore,  we immediately obtain equalities (\ref{eqn:J_rr_rs}), and its formal form in Lean 4: 

\begin{lstlisting}[frame = single]
lemma J_rr_linear (r : ℕ) : ∃ a : ℤ, J r r =
    2 * ∑' n : ℕ, 1 / ((n : ℝ) + 1) ^ 3 - a / (d (Finset.Icc 1 r)) ^ 3

lemma J_rs_linear {r s : ℕ} (h : r > s) : ∃ a : ℤ, J r s = a / d (Finset.Icc 1 r) ^ 3
\end{lstlisting}

By equalities (\ref{eqn:J_rr_rs}), we immediately complete the proof of Theorem \ref{maintheorem5.1}. \hfill $\qedsymbol$
\end{proof}

\subsection{Integer Sequence}
In order to construct an integer sequence of the form $\{a_n + b_n \zeta(3)\}$, we introduce the following integral
\[
\mathfrak{J}_n := \int_{(x,y)\in(0,1)^2} -P_n(x)P_n(y)\frac{\log (x y)}{1-x y}d\mu.
\]

The formal statement in Lean 4 is:
\begin{lstlisting}[frame = single]
noncomputable abbrev JJ (n : ℕ) : ℝ :=
    ∫ (x : ℝ × ℝ) in Set.Ioo 0 1 ×ˢ Set.Ioo 0 1,
    (-(x.1 * x.2).log / (1 - x.1 * x.2) * (shiftedLegendre n).eval x.1 * (shiftedLegendre n).eval x.2)
\end{lstlisting}
where  \lstinline{eval} denotes the evaluation of the shifted Legendre polynomial at a single point.

Since both  $P_n(x)$ and $P_n(y)$ are polynomials with  integer coefficients (\ref{eqn:legendre_int}), we can express   $P_n(x)$ and $P_n(y)$ as  finite sums $\sum\limits_{k=0}^{n}a_n x^k$ and $\sum\limits_{k=0}^{n}a_n y^k$, respectively, where $a_i \in \mathbb{Z}$.
By exchanging the order of the summation and the integral, we obtain
\[
    \mathfrak{J}_n = \sum\limits_{k=0}^{n}\sum\limits_{l=0}^{n} - a_ka_l\int_{(x,y)\in (0,1)^2} x^ky^l\frac{\log(x y)}{1-x y} d\mu = \sum\limits_{k=0}^{n}\sum\limits_{l=0}^{n} a_ka_lJ_{kl}
\]

By equalities (\ref{form:Jrs}) and $d_{\max\{k,l\}} | d_n$ for $k, l\leq n$, one obtains
\[
\mathfrak{J}_n:= \sum\limits_{k=0}^{n}\sum\limits_{l=0}^{n} a_ka_l \left(a_{kl}\zeta(3) + \frac{b_{kl}}{d_{\max\{k,l\}}^3}\right) \in \mathbb{Z}\zeta(3) + \frac{\mathbb{Z}}{d_n^3}.
\]

Hence, we  obtain a sequence  
\[\{ \mathfrak{J}_n \cdot d_n^3= a_n + b_n \zeta(3)\},\]
where $a_n, b_n \in \mathbb{Z}$ for all $n \in \mathbb{N}$. This sequence has been used in \cite[Theorem 2]{beukers1979note} for showing the irrationality of $\zeta(3)$. 

The function $\mathfrak{J}_n \cdot d_n^3$ is formally defined in Lean 4 as follows: 

\begin{lstlisting}[frame = single]
noncomputable abbrev fun1 (n : ℕ) : ℝ := (d (Finset.Icc 1 n)) ^ 3 * JJ n
\end{lstlisting}
We have successfully constructed the sequence  $\{a_n + b_n \zeta(3)\}$. The formal statement of the theorem  in Lean 4 is:  
\begin{lstlisting}[frame = single]
theorem linear_int (n : ℕ) : ∃ a b : ℕ → ℤ, fun1 n = a n + b n *
    (d (Finset.Icc 1 n) : ℤ) ^ 3  * ∑' n : ℕ, 1 / ((n : ℝ) + 1) ^ 3
\end{lstlisting}

\subsection{From Double Integral to Triple Integral}

Given the sequence $\{a_n + b_n \zeta(3)\}$, we need to prove two things: first, that the sequence is non-zero,  and second, that it tends to $0$ as $n \rightarrow \infty$.

First, we prove that the sequence is non-zero by demonstrating its positivity. Since $d_n^3$ is always positive, it suffices to show that $\mathfrak{J}_n$ is positive for all natural numbers  $n$.

However, determining the sign of the shifted Legendre polynomial on the interval $(0,1)$
 is not straightforward, so we prove the result by relating it to the triple integral.

\[
 \mathfrak{J}'_n := \int_{(x,y,z)\in(0,1)^3} \left(\frac{x(1-x)y(1-y)z(1-z)}{1-(1-y z)x}\right)^n \frac{1}{1-(1-yz)x} d\mu.
\]

It is formally defined in Lean 4 as follows:

\begin{lstlisting}[frame = single]
noncomputable abbrev JJ' (n : ℕ) : ℝ :=
    ∫ (x : ℝ × ℝ × ℝ) in Set.Ioo 0 1 ×ˢ Set.Ioo 0 1 ×ˢ Set.Ioo 0 1,
    (x.2.1 * (1 - x.2.1) * x.2.2 * (1 - x.2.2) * x.1 * (1 - x.1) /
    (1 - (1 - x.2.1 * x.2.2) * x.1)) ^ n / (1 - (1 - x.2.1 * x.2.2) * x.1)
\end{lstlisting}

Next, we demonstrate the following theorem: 
\begin{theorem}\label{JJ_eq_JJ'}
For any $n \in \mathbb{N}$,
\[
\mathfrak{J}_n  = \mathfrak{J}'_n.
\]
\end{theorem}
The formal statement in Lean 4 is:

\begin{lstlisting}[frame = single]
theorem JJ_eq_form (n : ℕ) : JJ n = JJ' n
\end{lstlisting}

We prove the equality by the following calculation
\begin{align}
        \mathfrak{J}_n &= \int_{(x,y)\in(0,1)^2} P_n(x)P_n(y)\left(\int_{0}^{1} \frac{1}{1 - (1 - xy)z} \, dz\right)d\mu \label{eqn:subs}\\
        &=\int_{0}^{1}\left(\int_{(x,y)\in(0,1)^2} P_n(x)P_n(y) \frac{1}{1 - (1 - xy)z} d\mu \right) \, dz \label{eqn:comm1}\\
        &=\int_{0}^{1}\left(\int_{(x,y)\in(0,1)^2} \frac{P_n(x)(xyz)^n(1-y)^n}{(1-(1-xy)z)^{n + 1}} d\mu \right) \, dz \label{eqn:double_integral_eq1_in}\\
        &=\int_{(x,y)\in(0,1)^2}\left(\int_{0}^{1} \frac{P_n(x)(1-z)^n(1-y)^n}{1-(1-xy)z}  \, dz\right) d\mu \label{eqn:comm2and}\\
        &=\int_{0}^{1} (1-z)^n \left( \int_{(x,y)\in(0,1)^2}\frac{(xyz(1-x)(1-y))^n}{(1-(1-xy)z)^{n+1}} d\mu \right) \, dz  \label{eqn:double_integral_eq3_in}\\
        &=\mathfrak{J}'_n \label{eqn:repeated_eq_double}
\end{align}

In the course of the above proof, it is necessary to establish the integrability of the following three functions $f_1(x,y,z), f_2(x,y,z)$ and $f_3(x,y,z)$, whose integrability is also  required  in the proofs of equality (\ref{eqn:comm1}), equality (\ref{eqn:comm2and}) and equality (\ref{eqn:double_integral_eq3_in}), respectively: 

\begin{align}
    f_1(x,y,z) =& P_n(x)P_n(y)\frac{1}{1-(1-x y)z}, \label{f1} \\ 
    f_2(x,y,z) =& P_n(x)(xyz)^n(1-y)^n\frac{1}{(1-(1-x y)z)^{n+1}}, \nonumber\\
    =& P_n(x)\left(\frac{xyz(1-y)}{1-(1-x y)z}\right)^n\frac{1}{1-(1-x y)z},\label{f2}\\
    f_3(x,y,z) =& P_n(x)(1-z)^n(1-y)^n\frac{1}{1-(1-x y)z}. \label{f3}
\end{align}

By equation (\ref{eqn:legendre_int}) and absolute value inequality, for any $n \in \mathbb{N}$ and $x \in (0,1)$, 

\[ P_n(x) \leq |P_n(x)| \leq \sum\limits_{k=0}^{n}\left|(-1)^k\binom{n}{k}\binom{n+k}{n}x^k\right| \leq \sum\limits_{k=0}^{n}\binom{n}{k}\binom{n+k}{n} \]

Let $C_n = \sum\limits_{k=0}^{n}\binom{n}{k}\binom{n+k}{n}$, then for any $n \in \mathbb{N}$ and $(x,y,z) \in (0,1)^3$, we have: 
\begin{align}
P_n(x)P_n(y) &\leqslant C_n^2, \nonumber \\
P_n(x)\left(\frac{xyz(1-y)}{1-(1-x y)z}\right)^n &\leqslant C_n\left(\frac{xyz(1-y)}{1-(1-xy)z}\right)^n \leqslant C_n, \nonumber \\
P_n(x)(1-z)^n(1-y)^n &\leqslant C_n. \nonumber
\end{align}

It is necessary for the following inequalities to hold:
$$
\frac{xyz(1-y)}{1-(1-xy)z} = \frac{xyz(1-y)}{1-z+xyz} \leqslant \frac{xyz(1-y)}{xyz} \leqslant 1-y \leqslant 1.
$$

Since a bounded factor does not affect integrability,  to prove that $f_1, f_2$ and $f_3$ have finite integrals over the unit cube $(0,1)^3$, it suffices to show that the function  $\frac{1}{1-(1-x y)z}$ has a finite integral over the unit cube $(0,1)^3$. This  will be proven in \cref{subsectionofeq20}.

\begin{subsubsection}{Proof of the Equality (\ref{eqn:subs})}

\begin{lemma}\label{integral1}
For $0 < a < 1$, one has
\begin{equation}\label{lemma6}
 \int_0^1 \frac{1}{1-(1-a) z} d z = -\frac{\ln a}{1-a},
 \end{equation}
\end{lemma}
which in Lean 4 is:

\begin{lstlisting}[frame = single]
lemma integral1 {a : ℝ} (ha : 0 < a) (ha1 : a < 1) :
    ∫ (z : ℝ) in (0)..1, 1 / (1 - (1 - a) * z) = - a.log / (1 - a)
\end{lstlisting}

\begin{proof} By substituting $(1-a)z=u$ in the integral (\ref{lemma6}), with $d u=(1-a) d z$, we obtain
$$
\begin{gathered}
\int_0^1 \frac{1}{1-(1-a) z} d z=\frac{1}{1-a} \int_0^{1-a} \frac{1}{1-u} d u =\frac{1}{1-a}[-\ln (1-u)]_0^{1-a}\\
=-\frac{1}{1-a}[\ln a-\ln 1]=-\frac{\ln a}{1-a}.
\end{gathered}$$ \hfill $\qedsymbol$
\end{proof}

Since $0<x,y<1$, it follows that  $0<xy<1$. 
To prove equality (\ref{eqn:subs}), we set  $a=xy$  in the integral (\ref{lemma6}).

\end{subsubsection}

\begin{subsubsection}{Proof of the Equality (\ref{eqn:comm1})}\label{subsectionofeq20}

To prove equality (\ref{eqn:comm1}), we rely on the \verb|MeasureTheory.integral_integral_swap| theorem from \texttt{Mathlib}, which enables the interchange of the order of integration.

\begin{lstlisting}[frame = single]
theorem MeasureTheory.integral_integral_swap ⦃f : α → β → E⦄
    (hf : Integrable (uncurry f) (μ.prod ν)) :
    ∫ x, ∫ y, f x y ∂ν ∂μ = ∫ y, ∫ x, f x y ∂μ ∂ν 
\end{lstlisting}

Here, \lstinline{uncurry f} denotes the transformation of $f$ into a function of type 
$\alpha \times \beta \to E$,  and \lstinline{μ.prod ν} refers to the product measure.

To apply this theorem, we need to demonstrate that $f_1(x,y,z)$ (\ref{f1}) is integrable on $(0,1)^3$. 
It is equivalent to the previously mentioned proof of the integrability of $\frac{1}{1-(1-x y)z}$. A more general version of the result is as follows:

\begin{lemma}\label{lem:double_eq_triple}
For $r, s \in \mathbb{N},$
\begin{equation}\label{lemma9eq}
\int_{(x,y)\in(0,1)^2}^- -\frac{\log (x y)}{1-x y}x^r y^s d\mu =  \int_{(x,y,z)\in(0,1)^3}^- \frac{1}{1-(1-xy)z}x^ry^s d\mu.
\end{equation}
\end{lemma}

The formal statement in Lean 4 is: 
\begin{lstlisting}[frame = single]
lemma JENN_eq_triple (r s : ℕ) : J_ENN r s =
    ∫⁻ (x : ℝ × ℝ × ℝ) in Set.Ioo 0 1 ×ˢ Set.Ioo 0 1 ×ˢ Set.Ioo 0 1,
    ENNReal.ofReal (1 / (1 - (1 - x.2.1 * x.2.2) * x.1) * x.2.1 ^ r * x.2.2 ^ s)
\end{lstlisting}

Before proving Lemma \ref{lem:double_eq_triple}, we present a result similar to Lemma \ref{integral1} as follows:

\begin{lemma}\label{lem:double_eq_triple2}
For $r, s \in \mathbb{N},$
\begin{equation}\label{lemma10} \int_{z \in (0,1)}^- \frac{1}{1-(1-xy)z}x^ry^s  = \frac{-\log (xy)}{1-xy}x^ry^s. 
\end{equation}
\end{lemma}
The formal statement in Lean 4 is: 
\begin{lstlisting}[frame = single]
lemma JENN_eq_triple_aux (x : ℝ × ℝ) (hx : x ∈ Set.Ioo 0 1 ×ˢ Set.Ioo 0 1) :
    ∫⁻ (w : ℝ) in Set.Ioo 0 1, ENNReal.ofReal (1 / (1 - (1 - x.1 * x.2) * w) * x.1 ^ r * x.2 ^ s) =
    ENNReal.ofReal (-Real.log (x.1 * x.2) / (1 - x.1 * x.2) * x.1 ^ r * x.2 ^ s)
\end{lstlisting}

\begin{proof} Since \verb|ENNReal.ofReal (x^r*y^s)| is independent of $z$,  we can factor it out of the integral on the left side of the equality (\ref{lemma10}). By  comparing  both sides of the equality (\ref{lemma10}), we now aim to  prove the following equation: 
\begin{equation}\label{JENN_eq_triple_aux'}
    \int_{z \in (0,1)}^- \frac{1}{1-(1-xy)z} = \frac{-\log (xy)}{1-xy}.
\end{equation}

The formal statement in Lean 4 is: 
\begin{lstlisting}[frame = single]
lemma JENN_eq_triple_aux' (x : ℝ × ℝ)
    (hx : x ∈ Set.Ioo 0 1 ×ˢ Set.Ioo 0 1) : ∫⁻ (w : ℝ) in Set.Ioo 0 1,
    ENNReal.ofReal (1 / (1 - (1 - x.1 * x.2) * w)) =
    ENNReal.ofReal (-Real.log (1 - (1 - x.1 * x.2)) / (1 - x.1 * x.2))
\end{lstlisting}

By applying  the theorem (\cref{lst:ofReal_integral_eq_lintegral_ofReal}) in~\cref{sec2} in reverse, we can take the \lstinline{ENNReal.ofReal} out of ``$\int^-$" symbolic.
However, to apply the theorem (\cref{lst:ofReal_integral_eq_lintegral_ofReal}), we need to check  the integrability condition first, i.e.,  
we need to prove the function $f(z) = \frac{1}{1-(1-x y)z}$ 
is integrable on $(0,1)$, where $(x,y) \in (0,1)^2$.
The antiderivative of $f(z)$ is 
\[g(z) = \frac{-\log (1-(1-x y)z)}{1-xy}.\]
Since $g$ is continuous and differentiable over the $[0,1]$, $f$ must be integrable on $(0,1)$.

Furthermore, we also need to 
prove that $f$ is greater than or equal to 0 almost everywhere which is the condition \verb|f_nn| in theorem (\cref{lst:ofReal_integral_eq_lintegral_ofReal}). By theorem  (\cref{ae_nonneg_restrict_of_forall_setIntegral_nonneg_inter}), it suffices to verify that   $f(z)$  is nonnegative at every point $(0,1)$, which is straightforward. 
 
Then, by setting  $a= xy$ in Lemma \ref{integral1}, the equality (\ref{JENN_eq_triple_aux'}) can be proved. \hfill $\qedsymbol$
\end{proof}

We now proceed to prove  Lemma \ref{lem:double_eq_triple}: 

\begin{proof} We transform the triple integral on the right-hand side of the equation into a single integral with respect to 
$z$, followed by a double integral over 
$x$ and $y$. 
We then compare both sides of the equation. For any $0 < x, y < 1 $, Lemma \ref{lem:double_eq_triple2} implies  Lemma \ref{lem:double_eq_triple}. \hfill $\qedsymbol$
\end{proof}

 By setting  $r=s=0$ in (\ref{lemma10}) and combining it with the previously proven equality (\ref{eq_lst7}), we can derive:

\begin{equation}\label{triple_eq_zeta3}
    \int_{(x,y,z)\in(0,1)^3}^- \frac{1}{1-(1-yz)x} d\mu = 2\cdot \sum_{n \in \mathbb{N}} \frac{1}{(n + 1)^3}.
\end{equation}

The formal statement in Lean 4 is:
\begin{lstlisting}[frame = single]
∫⁻ (x : ℝ × ℝ × ℝ) in Set.Ioo 0 1 ×ˢ Set.Ioo 0 1 ×ˢ Set.Ioo 0 1, ENNReal.ofReal (1 / (1 - (1 - x.2.1 * x.2.2) * x.1)) =
ENNReal.ofReal (2 * ∑' (n : ℕ), 1 / ((n : ℝ) + 1) ^ 3)
\end{lstlisting}

At this point, since $2\cdot \sum_{n \in \mathbb{N}} \frac{1}{(n + 1)^3}$ is a real number, it naturally follows that $\frac{1}{1-(1-x y)z}$ has a finite integral 
on $(0,1)^3$. Therefore, we can obtain the integrability of \( f_1 \), \( f_2 \) and \( f_3 \) on \( (0,1)^3 \), where the integrability of \( f_1 \) proves the equality (\ref{eqn:comm1}).

\end{subsubsection}

\begin{subsubsection}
{Proof of the Equality (\ref{eqn:double_integral_eq1_in})}

To prove equality \ref{eqn:double_integral_eq1_in}, it suffices to show that the following two functions are equal for any $0 < z < 1$.

\begin{lemma}\label{eqn:double_integral_eq1}
For $0 < z < 1$, one has
\begin{equation}\label{eqlemma11}
\int_{(x,y)\in(0,1)^2} P_n(x)P_n(y) \frac{1}{1 - (1 - xy)z} d\mu= \int_{(x,y)\in(0,1)^2} \frac{P_n(x)(xyz)^n(1-y)^n}{(1-(1-xy)z)^{n + 1}} d\mu
\end{equation}
\end{lemma}

The formal statement in Lean 4 is:
\begin{lstlisting}[frame = single]
lemma double_integral_eq1 (n : ℕ) (z : ℝ) (hz : z ∈ Set.Ioo 0 1) :
    ∫ (x : ℝ × ℝ) in Set.Ioo 0 1 ×ˢ Set.Ioo 0 1, eval x.1 (shiftedLegendre n) * eval x.2 (shiftedLegendre n) * (1 / (1 - (1 - x.1 * x.2) * z)) =
    ∫ (x : ℝ × ℝ) in Set.Ioo 0 1 ×ˢ Set.Ioo 0 1, eval x.1 (shiftedLegendre n) * (x.1 * x.2 * z) ^ n * (1 - x.2) ^ n / (1 - (1 - x.1 * x.2) * z) ^ (n + 1) 
\end{lstlisting}

\begin{proof} For   $0 < z < 1$, let 
\[f(x,y) = \frac{P_n(x)P_n(y)}{1-(1-x y)z},\]
and 
\[g(x,y) = \frac{P_n(x)(xyz)^n(1-y)^n}{(1-(1-x y)z)^{n+1}}. \]

 Since $0<z<1$, the functions $f$ and $g$ have no singular points within the unit square $[0,1]^2$, and both of them are continuously differentiable. Moreover, $f$ and $g$   are integrable on the unit square as they are continuous functions on the compact set $[0, 1]^2$. Hence, to prove  equality (\ref{eqlemma11}), one can   convert the two integrals  $\int_{(x,y)\in(0,1)^2} f(x,y)$ and $ \int_{(x,y)\in(0,1)^2} g(x,y)$
  into repeated integrals and compare the integrands for $x$. In other words, it suffices to prove the following equality: 
\[ \int_0^1 \frac{P_n(y)}{1-(1-x y)z} \, d y = \int_0^1 \frac{(xyz)^n(1-y)^n}{(1-(1-x y)z)^{n+1}} \, d y\]
where $x,z \in (0,1)$. This follows from Lemma \ref{int_shifted_legendre}. \hfill $\qedsymbol$
\end{proof}

\end{subsubsection}

\begin{subsubsection}
{Proof of the Equality (\ref{eqn:comm2and})}

We begin by swapping the order of the double integral for $x$ and $y$
 and the single integral for 
$z$.  This step follows the pattern of equality (\ref{eqn:subs}) and requires the integrability of $f_2(x,y,z)$ (\ref{f2}), which has been proven in \cref{subsectionofeq20}.

To start, we prove the following lemma:

\begin{lemma}\label{eqn:double_integral_eq2}
For $0<x,y<1$, one has
\begin{equation}\label{eqlemma8}
\int_{0}^{1} \frac{P_n(x)(xyz)^n(1-y)^n}{(1-(1-xy)z)^{n + 1}} \, dz = \int_{0}^{1} \frac{P_n(x)(1-z)^n(1-y)^n}{1-(1-xy)z} \, dz
\end{equation}
\end{lemma}

The formal statement in Lean 4 is:
\begin{lstlisting}[frame = single]
lemma double_integral_eq2 (n : ℕ) (x : ℝ × ℝ) (hx : x ∈ Set.Ioo 0 1 ×ˢ Set.Ioo 0 1) : ∫ (z : ℝ) in Set.Ioo 0 1, eval x.1 (shiftedLegendre n) * (x.1 * x.2 * z) ^ n * (1 - x.2) ^ n / (1 - (1 - x.1 * x.2) * z) ^ (n + 1) = ∫ (z : ℝ) in Set.Ioo 0 1, eval x.1 (shiftedLegendre n) * (1 - z) ^ n * (1 - x.2) ^ n / (1 - (1 - x.1 * x.2) * z)
\end{lstlisting}

\begin{proof} After removing  the factors $P_n(x)(1-y)^n $ on both sides of equality (\ref{eqlemma8}) that are unrelated to the integral variable $z$,  it is sufficient to show
\begin{equation}
\int_0^1 \frac{(xyz)^n}{(1-(1-x y)z)^{n+1}} \, dz = \int_0^1 \frac{(1-z)^n}{1-(1-x y)z} \, dz
\label{eqn:36}
\end{equation}

If a function $f$ has continuous derivative $f'$ on $[a, b]$, and $g$ is continuous, then by substituting  $u = f (x)$, we have the equality:
\begin{equation}\label{integral_compositon}
    \int_a^b (g \circ f) (x)f'(x) \, dx= \int_{f(a)}^{f(b)} g(u) \, du,
\end{equation}
which follows from the theorem (\cref{change_of_variables}) in \cref{sec2}. 

Let  $f(z) = \frac{1-z}{1-(1-x y)z}$, with $f(0) = 1$ and $f(1)=0$. Then for the right side of equality (\ref{eqn:36}), it takes the following form:
\begin{equation}\label{eq322} \int_0^1 \frac{(1-z)^n}{1-(1-x y)z} \, dz = \int_0^1 \frac{(1-w)^n}{1-(1-x y)w} \, dw = \int_{f(1)}^{f(0)} \frac{(1-w)^n}{1-(1-x y)w} \, dw.
\end{equation}

By substituting $w = f(z)$ into equality (\ref{eq322}) and  use equality (\ref{integral_compositon}), followed by straightforward calculations,   we have the equality  (\ref{eqn:36}):
\begin{align*}
    \int_{f(1)}^{f(0)} \frac{(1-w)^n}{1-(1-x y)w} \, dw &= - \int_{f(0)}^{f(1)} \frac{(1-w)^n}{1-(1-x y)w} \, dw \\
    &= - \int_0^1 \frac{(1-f(z))^n}{1-(1-x y)f(z)}f'(z) \, dz \\
    &= - \int_0^1 (\frac{xyz}{1-(1-x y)z})^n\frac{1-(1-x y)z}{xy}\frac{-xy}{(1-(1-x y)z)^2} \, dz \\
    &= \int_0^1 \frac{(xyz)^n}{(1-(1-x y)z)^{n+1}} \, dz. 
\end{align*} \hfill $\qedsymbol$
\end{proof}

\end{subsubsection}

\begin{subsubsection}
{ Proof of the Equality (\ref{eqn:double_integral_eq3_in})}

To show the equality in (\ref{eqn:double_integral_eq3_in}), we first demonstrate the following:
\begin{align*}
    &\int_{(x,y)\in(0,1)^2}\left(\int_{0}^{1} \frac{P_n(x)(1-z)^n(1-y)^n}{1-(1-xy)z}  \, dz\right) d\mu \\
=&\int_{0}^{1} (1-z)^n \left(\int_{(x,y)\in(0,1)^2}\frac{P_n(x)(1-y)^n}{1-(1-xy)z} d\mu \right) \, dz. 
\end{align*}
It is sufficient to exchange the order of integration and use the integrability of 
$f_3(x,y,z)$ (as shown in  (\ref{f3})), which was established in  \cref{subsectionofeq20}.

We then proceed to formalize the following lemma:
\begin{lemma}\label{eqn:double_integral_eq3}
For $0<z<1$, one has
\[
\int_{(x,y)\in(0,1)^2}\frac{P_n(x)(1-y)^n}{1-(1-xy)z} d\mu = \int_{(x,y)\in(0,1)^2}\frac{(xyz(1-x)(1-y))^n}{(1-(1-xy)z)^{n+1}} d\mu.
\]
\end{lemma}

The formal statement in Lean 4 is:
\begin{lstlisting}[frame = single]
lemma double_integral_eq3 (n : ℕ) (z : ℝ) (hz : z ∈ Set.Ioo 0 1) : ∫ (x : ℝ × ℝ) in Set.Ioo 0 1 ×ˢ Set.Ioo 0 1,
    (x.1 * x.2 * z * (1 - x.1) * (1 - x.2)) ^ n / (1 - (1 - x.1 * x.2) * z) ^ (n + 1) 
    = ∫ (x : ℝ × ℝ) in Set.Ioo 0 1 ×ˢ Set.Ioo 0 1, eval x.1 (shiftedLegendre n) * (1 - x.2) ^ n / (1 - (1 - x.1 * x.2) * z)
\end{lstlisting}

\begin{proof} The proof is similar to that  of equality (\ref{eqn:double_integral_eq1_in}), achieved by interchanging the role of $x$ and $y$. \hfill $\qedsymbol$
\end{proof}

\end{subsubsection}

\begin{subsubsection}
{ Proof of the Equality (\ref{eqn:repeated_eq_double})} 

First, we move $(1-z)^2$ inside the integral ``$\int_{(x,y)\in(0,1)^2}$". Next, we need to prove that a repeated integral is equal to a triple integral, as stated in Theorem (\cref{multiple_integral_equal_repeated_integral}. This requires proving the following lemma regarding integrability:
\begin{lemma}\label{integrable_lemma}
For any $n \in \mathbb{N}$, the function
\[
\frac{(xyz(1-x)(1-y)(1-z))^n}{(1-(1-xy)z)^{n+1}} =\left(\frac{xyz(1-x)(1-y)(1-z)}{1-(1-xy)z}\right)^n\frac{1}{1-(1-xy)z}
\]
is integrable on $(0,1)^3$.
\end{lemma}
The formal statement in Lean 4 is:
\begin{lstlisting}[frame = single]
lemma integrableOn_JJ' (n : ℕ) : MeasureTheory.Integrable
    (fun (x : ℝ × ℝ × ℝ) ↦ (x.2.1 * (1 - x.2.1) * x.2.2 * (1 - x.2.2) * x.1 * (1 - x.1) / (1 - (1 - x.2.1 * x.2.2) * x.1)) ^ n /
    (1 - (1 - x.2.1 * x.2.2) * x.1)) (MeasureTheory.volume.restrict (Set.Ioo 0 1 ×ˢ Set.Ioo 0 1 ×ˢ Set.Ioo 0 1))
\end{lstlisting}

We define a $[0, \infty]$ valued function $\mathcal{J}_n$:
\[
\mathcal{J}_n := \int^{-}_{(x,y,z) \in (0,1)^3} \left(\frac{xyz(1-x)(1-y)(1-z)}{1-(1-xy)z}\right)^n\frac{1}{1-(1-xy)z} d\mu.
\] 
The formal definition  in Lean 4 is:
\begin{lstlisting}[frame = single]
noncomputable abbrev JJENN (n : ℕ) : ENNReal := ∫⁻ (x : ℝ × ℝ × ℝ) in Set.Ioo 0 1 ×ˢ Set.Ioo 0 1 ×ˢ Set.Ioo 0 1, ENNReal.ofReal ((x.2.1 * (1 - x.2.1) * x.2.2 * (1 - x.2.2) * x.1 * (1 - x.1) / (1 - (1 - x.2.1 * x.2.2) * x.1)) ^ n / (1 - (1 - x.2.1 * x.2.2) * x.1))
\end{lstlisting}

\begin{proof}[Proof of Lemma \ref{integrable_lemma}]
We need to prove that $\mathcal{J}_n$ has a finite integral.   It suffices  to show that
\begin{equation}\label{ineq35}
\mathcal{J}_n \leq 2 \left(\frac{1}{24}\right)^n \sum\limits_{n \in \mathbb{N}} \frac{1}{(n+1)^3} = 2\left( \frac{1}{24}\right)^n \zeta(3).
\end{equation}

The formal statement in Lean 4 is:
\begin{lstlisting}[frame = single]
lemma JJENN_upper (n : ℕ) : JJENN n ≤ ENNReal.ofReal
    (2 * (1 / 24) ^ n * ∑' n : ℕ, 1 / ((n : ℝ) + 1) ^ 3)
\end{lstlisting}

To prove  inequality (\ref{ineq35}), we first  demonstrate that that for all $x,y,z \in (0, 1)$, the following holds:
\begin{equation}\label{ineq352}
\frac{x(1-x)y(1-y)z(1-z)}{1-(1-xy)z} < \frac1{24}.
\end{equation}
The formal statement  in Lean 4 is: 
\begin{lstlisting}[frame = single]
lemma bound' (x y z : ℝ) (x0 : 0 < x) (x1 : x < 1) (y0 : 0 < y) (y1 : y < 1) (z0 : 0 < z) (z1 : z < 1) :
    x * (1 - x) * y * (1 - y) * z * (1 - z) / (1 - (1 - x * y) * z) < (1 / 24 : ℝ) 
\end{lstlisting}

From the inequality
\[ 1-(1-xy)z = 1-z + xyz \geqslant 2\sqrt{1-z}\sqrt{xyz},  \]
we can deduce that for  $x,y,z \in (0, 1)$, the following holds:  
\begin{align*}
    \frac{x(1-x)y(1-y)z(1-z)}{(1-(1-xy)z)} \leqslant& \frac{x(1-x)y(1-y)z(1-z)}{2\sqrt{1-z}\sqrt{xyz}}\\
    =&\frac{\sqrt{x}(1-x)\sqrt{y}(1-y)\sqrt{z}\sqrt{1-z}}{2}.
\end{align*}

For $z\in (0,1)$, the maximum value of $\sqrt{z}\sqrt{1-z} $ is attained  at $z=\frac{1}{2}$.
For $y \in (0,1)$, we have 
\[y(1-y)^2 - \frac{4}{27} = (y - \frac{4}{3})(y - \frac{1}{3})^2 \leqslant 0.\]
Hence, for $y \in (0,1)$, we have
\[ \sqrt{y}(1-y) = \sqrt{y(1-y)^2} \leqslant \sqrt{\frac{4}{27}} \leqslant \sqrt{\frac{4}{25}} = \frac{2}{5}. \]
Finally,   inequality (\ref{ineq352}) is proven
\begin{align*}
    \frac{x(1-x)y(1-y)z(1-z)}{(1-(1-xy)z)} \leqslant \frac{2}{5}\cdot\frac{2}{5}\cdot\frac{1}{2}\cdot\frac{1}{2} = \frac{1}{25} < \frac{1}{24}
\end{align*}

Therefore,   inequality (\ref{ineq35}) holds:
\begin{align*}
    & \mathcal{J}_n =\int^{-}_{(x,y,z) \in (0,1)^3} \left(\frac{xyz(1-x)(1-y)(1-z)}{1-(1-xy)z}\right)^n\frac{1}{1-(1-xy)z} d\mu\\
    \leq& \int^{-}_{(x,y,z) \in (0,1)^3} \left(\frac{1}{24}\right)^n\frac{1}{1-(1-xy)z} d\mu\\
    = & \left(\frac{1}{24}\right)^n \int^{-}_{(x,y,z) \in (0,1)^3} \frac{1}{1-(1-xy)z} d\mu \\
    = & \left(\frac{1}{24}\right)^n 2 \zeta(3).
\end{align*}
Consequently, we formalize \lstinline{JJENN_upper} by using equation (\ref{lemma9eq}) and equation (\ref{triple_eq_zeta3}) with $r = 0$. 

\hfill $\qedsymbol$
\end{proof}

\end{subsubsection}

\begin{subsection}{Positive Sequences Converging to Zero}

First, we establish that $\mathfrak{J}_n$ is always positive.
\begin{theorem}\label{thm:pos}
For any $n \in \mathbb{N}$,
\[\mathfrak{J}_n > 0. \]
\end{theorem}

The formal statement in Lean 4 is:
\begin{lstlisting}[frame = single]
    theorem JJ_pos (n : ℕ) : 0 < JJ n
\end{lstlisting}

\begin{proof}
By Theorem \ref{JJ_eq_JJ'}, it suffices to  prove that $\mathfrak{J}^{\prime}_n$  is always positive.

We apply the theorem (\cref{lst:integral_pos_iff_support_of_nonneg_ae}) from  \texttt{Mathlib} in our proof.  It suffices to show that the measure of the support of the function is larger than $0$.

In this case, the measure $\mu$ is the measure on $\mathbb{R}$ restricted to $(0,1)^3$. Let \[f(x,y,z) = \left(\frac{x(1-x)y(1-y)z(1-z)}{1-(1-x y)z}\right)^n\frac{1}{1-(1-x y)z}. \]

We can verify that   for any $(x,y,z) \in (0,1)^3$,
we have 
\[f(x,y,z) > 0.\]
Therefore, 
$(0,1)^3 \subset \operatorname{supp} f$, and the measure of $\operatorname{supp} f$ must be greater than that of  $(0,1)^3$, which is positive. 

The integrability condition  \verb|hfi| is provided by  Lemma \ref{integrable_lemma}. As for the condition  \verb|hf| in theorem (\cref{lst:integral_pos_iff_support_of_nonneg_ae}), we need to show $f(x,y,z)$ is almost everywhere greater than or equal to $0$. Similar to the proof of Lemma \ref{lem:double_eq_triple2}, this can be established by using Theorem (\cref{ae_nonneg_restrict_of_forall_setIntegral_nonneg_inter}), Lemma \ref{integrable_lemma}, and the fact that $f$ is positive on $(0,1)^3$. \hfill $\qedsymbol$
\end{proof}

By theorem (\cref{lst:ofReal_integral_eq_lintegral_ofReal}), we can move \verb|ENNReal.ofReal| from inside to outside the integral. 
We obtain the following theorem:

\begin{theorem}\label{thm:bound}
For all $n \in \mathbb{N}$,
\[ 
\mathfrak{J}_n \leq 2\cdot \left(\frac{1}{24}\right)^n \sum\limits_{n \in \mathbb{N}} \frac{1}{(n+1)^3} = 2\cdot\left( \frac{1}{24}\right)^n \zeta(3).
\]
\end{theorem}

 The formal statement in Lean 4 is: 
\begin{lstlisting}[frame = single]
Theorem JJ_upper (n : ℕ) :
    JJ n ≤ 2 * (1 / 24) ^ n * ∑' n : ℕ, 1 / ((n : ℝ) + 1) ^ 3
\end{lstlisting}

Next, we demonstrate that the sequence converges to zero. 
\begin{theorem}\label{thm5}

The sequence $\{a_n + b_n \zeta(3)\}$ tends to $0$ when $n \rightarrow \infty$.
\end{theorem}

 The formal statement in Lean 4 is:
\begin{lstlisting}[frame = single]
theorem fun1_tendsto_zero : Filter.Tendsto (fun n ↦ ENNReal.ofReal (fun1 n)) Filter.atTop (nhds 0)
\end{lstlisting}

\begin{proof}
According to Theorem \ref{thm:bound}, we have
\[
\mathfrak{J}_n\cdot d_n^3 \leq 2\left( \frac{1}{24}\right)^n d_n^3 \zeta(3).
\]
Since $2\zeta(3)$ is constant, we can analyze the asymptotic behavior of $d_n^3$ for sufficiently large $n$. Using Theorem \ref{pi_alt} and equation (\ref{upperbounddn}),  for sufficiently large $n$, we have  $d_n^3 \leq 21^n$.

Therefore, we can conclude that for sufficiently large $n$, the following holds:
\[
\mathfrak{J}_n\cdot d_n^3 \leq \left( \frac{21}{24}\right)^n 2 \zeta(3).
\]
When $n \rightarrow \infty$, $\left(\frac{21}{24}\right)^n \rightarrow 0$. Hence, the sequence $\mathfrak{J}_n\cdot d_n^3$ tends to $0$, which implies that  the sequence $\{a_n + b_n \zeta(3)\}$ converges to $0$ as $n \rightarrow \infty$. \hfill $\qedsymbol$
\end{proof}

\subsection{Irrationality of $\zeta(3)$}

Finally, we establish the irrationality of  $\zeta(3)$:

\begin{theorem}
    $\zeta(3)$ is irrational.
\end{theorem}   

The formal statement in Lean 4 is: 
\begin{lstlisting}[frame = single]
theorem zeta_3_irratoinal : ¬ ∃ r : ℚ, r = riemannZeta 3
\end{lstlisting}

\begin{proof} Assume, for the sake of contradiction, that $\zeta(3) = \frac{p}{q}$, where  $\gcd(p,q)=1$ and $p,q>0$. Then by Theorem \ref{thm5}, we have $qa_n + pb_n \rightarrow 0$ as $n \rightarrow \infty$, since $q$ is a constant. 
Theorem \ref{thm:pos} states that $a_n + b_n\zeta(3)>0$ and $q>0$, which implies  that  $qa_n + pb_n > 0$.  Furthermore, since $a_n, b_n$ are integers,  $qa_n + pb_n \in \mathbb{Z}$. Therefore, we have $qa_n + pb_n \geqslant 1$ for all $n \in \mathbb{N}$. This leads to a contradiction, thereby implying that   $\zeta(3)$ is irrational. \hfill $\qedsymbol$
\end{proof}
\end{subsection}

\section{Conclusion}

Our work delivers the first complete formalization of the irrationality of \(\zeta(3)\) in Lean~4. To transform Beukers’ informal proof into a fully formal proof in Lean, we carefully adjusted and refined his arguments to meet the strict requirements of formalization. This process allowed us to bridge the gaps in the original proof and produce a complete, machine-verified demonstration within Lean’s framework.
We formally define the shifted Legendre polynomials and prove their fundamental properties. Additionally, we provide  the first formal proof in Lean~4 of a version of the Prime Number Theorem with an error term which is stronger than what had previously been formalized. This achievement significantly advances Lean’s analytical capabilities in number theory.







\acknowledgements{\rm Junqi Liu and  Lihong  Zhi are supported by the National Key R$\&$D Program of China 2023YFA1009401 and the Strategic Priority Research Program of Chinese Academy of Sciences under Grant XDA0480501. The author would like to thank Kevin Buzzard and Shaoshi Chen for helpful discussions and suggestions.}


\end{document}